\documentclass[12pt]{amsart}
\usepackage{amsmath,amsthm,amsfonts,amssymb,mathrsfs}%\usepackage{paralist,nicefrac}
\usepackage[all]{xy}
\usepackage{eufrak}
\usepackage{amssymb}
\usepackage{latexsym}
\usepackage{amsmath,amsthm}
\usepackage{aeguill}
\usepackage{amssymb}
\usepackage{mathrsfs}
\usepackage{hyperref}
\usepackage{etoolbox}
\usepackage{enumerate}
\usepackage{ textcomp }
\usepackage[total={6.5in,8.75in}, top=1.2in, left=1.0in, includefoot]{geometry}

\makeatletter
\pretocmd{\chapter}{\addtocontents{toc}{\protect\addvspace{15\p@}}}{}{}
\pretocmd{\section}{\addtocontents{toc}{\protect\addvspace{3\p@}}}{}{}
\makeatother

\makeatletter
\def\@tocline#1#2#3#4#5#6#7{\relax
  \ifnum #1>\c@tocdepth % then omit
  \else
    \par \addpenalty\@secpenalty\addvspace{#2}%
    \begingroup \hyphenpenalty\@M
    \@ifempty{#4}{%
      \@tempdima\csname r@tocindent\number#1\endcsname\relax
    }{%
      \@tempdima#4\relax
    }%
    \parindent\z@ \leftskip#3\relax \advance\leftskip\@tempdima\relax
    \rightskip\@pnumwidth plus4em \parfillskip-\@pnumwidth
    #5\leavevmode\hskip-\@tempdima
      \ifcase #1
       \or\or \hskip .5em \or \hskip 1em \else \hskip 1.5em \fi%
      #6\nobreak\relax
    \dotfill\hbox to\@pnumwidth{\@tocpagenum{#7}}\par
    \nobreak
    \endgroup
  \fi}
\makeatother
\newcommand{\G}{\mathbb{G}}
\newcommand{\C}{\mathbb{C}}
\newcommand{\N}{\mathbb{N}}
\newcommand{\Z}{\mathbb{Z}}
\newcommand{\R}{\mathbb{R}}
\newcommand{\Q}{\mathbb{Q}}
\newcommand{\F}{\mathbb{F}}

\newcommand{\X}{\mathbb{X}}

\newcommand{\cG}{\mathcal{G}}

\newcommand{\Gal}{\operatorname{Gal}}

\newcommand{\rk}{\operatorname{rk}}

\newcommand{\ad}{\operatorname{ad}}

\renewcommand{\sc}{\operatorname{sc}}
\renewcommand{\ss}{\operatorname{ss}}
\newcommand{\der}{\operatorname{der}}

\newcommand{\Aut}{\operatorname{Aut}}
\newcommand{\End}{\operatorname{End}}

\newcommand{\GL}{\mathrm{GL}}

\newcommand{\SL}{\mathrm{SL}}
\newcommand{\PSL}{\mathrm{PSL}}

\newcommand{\bB}{\mathbf{B}}
\newcommand{\bC}{\mathbf{C}}

\newcommand{\bG}{\mathbf{G}}
\newcommand{\bH}{\mathbf{H}}
\newcommand{\bM}{\mathbf{M}}
\newcommand{\bN}{\mathbf{N}}
\newcommand{\bS}{\mathbf{S}}
\newcommand{\bT}{\mathbf{T}}

\newcommand{\bY}{\mathbf{Y}}
\newcommand{\LT}{\mathfrak{g}}

\newtheorem{thm}{Theorem}[section]

\newtheorem{cor}[thm]{Corollary}
\newtheorem{corollary}[thm]{Corollary}
\newtheorem{prop}[thm]{Proposition}
\newtheorem{lemma}[thm]{Lemma}

\theoremstyle{definition}
\newtheorem{definition}[thm]{Definition}
\newtheorem{remark}[thm]{Remark}

\newtheorem{innercustomthm}{Theorem}
\newenvironment{customthm}[1]
  {\renewcommand\theinnercustomthm{#1}\innercustomthm}
  {\endinnercustomthm}

\newtheorem{innercustomcor}{Corollary}
\newenvironment{customcor}[1]
  {\renewcommand\theinnercustomcor{#1}\innercustomcor}
  {\endinnercustomcor}

\newtheorem{innercustomhyp}{Hypothesis}
\newenvironment{customhyp}[1]
  {\renewcommand\theinnercustomhyp{#1}\innercustomhyp}
  {\endinnercustomhyp}
	
\begin{document}

\title{On the rationality of certain type A Galois representations}

\author{Chun Yin Hui}

\thanks{The present project was supported by the National Research Fund, Luxembourg, and cofunded under the Marie Curie Actions of the European Commission (FP7-COFUND)}

\begin{abstract} 
Let $X$ be a complete smooth variety defined over a 
number field $K$ and $i$ an integer. 
The absolute Galois group $\Gal_K$ of $K$ acts on the $i$th \'etale cohomology group 
$H^i_{\mathrm{\acute{e}t}}(X_{\bar K},\Q_\ell)$ for all primes $\ell$, producing a system of $\ell$-adic representations $\{\Phi_\ell\}_\ell$. 
The conjectures of Grothendieck, Tate, and Mumford-Tate predict that the identity component 
of the algebraic monodromy group of $\Phi_\ell$ 
 admits a common reductive $\Q$-form for all $\ell$
 if $X$ is projective. 
Denote by $\Gamma_\ell$ and $\bG_\ell$ respectively the monodromy group and the algebraic
monodromy group of $\Phi_\ell^{\ss}$, the semisimplification of $\Phi_\ell$. 
Assuming that $\bG_{\ell_0}$ satisfies a group theoretic condition for some prime $\ell_0$ (Hypothesis \ref{A}), 
we construct a connected quasi-split $\Q$-reductive group $\bG_\Q$ which is a common $\Q$-form of 
$\bG_\ell^\circ$ for all sufficiently large $\ell$. 
Let $\mathbf{G}_\Q^{\sc}$ be the universal cover of the derived group of $\mathbf{G}_\Q$.
As an application, we prove that the monodromy group $\Gamma_\ell$ is big in the sense that 
$\Gamma_\ell^{\sc}\cong \bG_\Q^{\sc}(\Z_\ell)$ for all sufficiently large $\ell$.
\end{abstract}

\maketitle
{\tableofcontents}

\section{Introduction}

Let $X$ be a complete, smooth variety defined over a number field $K$ and $i$ an integer belonging to $[0,2\dim X]$. 
The absolute Galois group $\Gal_K:=\Gal(\bar K/K)$ acts on the $i$th $\ell$-adic \'etale cohomology group 
$V_\ell:=H^i_{\mathrm{\acute{e}t}}(X_{\bar K},\Q_\ell)$ for every ordinary prime $\ell$. 
We obtain by Deligne \cite{De74} a \emph{strictly compatible system} of 
$\ell$-adic representations
\begin{equation}\label{1}
\{\Phi_\ell:\Gal_K\to \GL(V_\ell)\}_\ell
\end{equation}
 in the sense of Serre \cite{Se98}. 
The \emph{algebraic monodromy group at $\ell$} denoted by $\mathbf{M}_\ell$, 
is the Zariski closure of the \emph{monodromy group} $\Phi_\ell(\Gal_K)$ in $\GL_{V_\ell}$. Let $\bM_\ell^\circ$ be the identity component of $\bM_\ell$ and $k$ the dimension of $V_\ell$ for all $\ell$.

Choose an embedding $K\hookrightarrow\C$.
If $X$ is projective, then $X(\C)$ is a compact K$\ddot{\mathrm{a}}$hler manifold and the singular cohomology group 
$V:=H^i(X(\C),\Q)$ carries a $\Q$-Hodge structure. 
Denote the Mumford-Tate group of $V$ by $\mathrm{MT}(V)$, which is 
a connected reductive subgroup of $\GL_{V}$. 
The celebrated conjectures of Grothendieck, Tate,
 and Mumford-Tate imply that 
\begin{equation}\label{conj}
\mathbf{M}_\ell^\circ\cong\mathrm{MT}(V)\times_\Q\Q_\ell
\end{equation}
via the comparison isomorphism between $V_\ell$ and $V\otimes_\Q\Q_\ell$ 
for all $\ell$ (see \cite{Ta65}, \cite[$\mathsection3$]{Se94}).
This is equivalent to saying that the inclusion $\mathrm{MT}(V)\subset \GL_{V}$ 
is a $\Q$-form of $\mathbf{M}_\ell^\circ\subset \GL_{V_\ell}$ for all $\ell$. 
It follows that 
the \emph{absolute root datum} of $\bM_\ell^\circ$, i.e., the \emph{root datum} of $\bM_\ell^\circ\times_{\Q_\ell}\bar\Q_\ell$, is 
independent of $\ell$.

Since $\Phi_\ell$ is conjecturally semisimple 
 and our methods only handle semisimple representations,
we denote, for all $\ell$, the \emph{semisimplification} of $\Phi_\ell$ by $\Phi_\ell^{\ss}$. 
We say that $\{\Phi_\ell^{\ss}\}_\ell$ is the semisimplification of the system (\ref{1}).
Let $\Gamma_\ell$ and $\bG_\ell$ be respectively the monodromy group (Galois image) and the algebraic monodromy group of $\Phi_\ell^{\ss}$. 
Let $\mathbf{U}_\ell$ be the unipotent radical of $\bM_\ell$. The following short exact sequence holds,
$$1\to \mathbf{U}_\ell\to \bM_\ell\to \bG_\ell\to 1.$$
Since we are only concerned about $\bG_\ell^\circ$ and there exists a finite extension $K^\mathrm{conn}$ of $K$
which is the smallest extension 
such the Zariski closure of $\Phi_\ell^{\ss}(\Gal_{K^{\mathrm{conn}}})$ in $\GL_{V_\ell^{\ss}}$ 
is $\bG_\ell^\circ$ for all $\ell$ \cite[$\mathsection2.2.3$]{Se84},
we once and for all assume that the field $K$ is chosen large enough such that $\bG_\ell$ is connected for all $\ell$
\footnote{Larsen-Pink presented a purely field theoretic construction
of $K^\mathrm{conn}$ in \cite{LP97}}. 

We embed $\Q_\ell$ in $\C$ and
let $\mathfrak{g}_\ell$ be 
the Lie algebra of $\bG_\ell\times_{\Q_\ell}\C$ for all $\ell$. 
The representation $\Phi_\ell^{\ss}$ and the algebraic monodromy group $\bG_\ell$ are said to be \emph{of type A} if 
every simple factor of $\mathfrak{g}_\ell$ is equal to $A_n:=\mathfrak{sl}_{n+1,\C}$ for some $n$.
This definition is independent of 
the choice of embedding $\Q_\ell\hookrightarrow\C$ and is equivalent to the one in \cite{HL16}.
Type A representations provide supporting evidence for (\ref{conj}).
For example, we showed in \cite{HL16} that for all sufficiently large $\ell$, $\bG_\ell$ is quasi-split if it is of type A.
Also, it follows from the
the main theorems of \cite{Hui13} that the
complex reductive Lie algebra $\mathfrak{g}_\ell$ 
is independent of $\ell$ if the following hypothesis is satisfied (see $\mathsection$\ref{l-adic}).

\begin{customhyp}{A}\label{A}
There exists a prime $\ell_0$ such that the following conditions hold for $\mathfrak{g}_{\ell_0}$:
\begin{enumerate}[($i$)]
\item $\mathfrak{g}_{\ell_0}$ has at most one $A_4$ simple factor;
\item if $\mathfrak{q}$ is a simple factor of $\mathfrak{g}_{\ell_0}$, then $\mathfrak{q}$ is of type $A_n$ for some $n\in \N\backslash\{1,2,3,5,7,8\}$.
\end{enumerate}
\end{customhyp}

\noindent\textbf{Example:} $\mathfrak{g}_{\ell_0}=A_4\oplus A_6\oplus A_9\oplus A_9\oplus Z$, where $Z$ is abelian.\\

This paper is motivated by the conjectural isomorphism (\ref{conj}) for all $\ell$. 
Suppose $X$ is an abelian variety. Then the semisimplicity of (\ref{1}) is established by Faltings \cite{Fa83} and (\ref{conj}) is the Mumford-Tate conjecture for abelian varieties \cite{Mu66}, which has been studied by many people (see \cite[$\mathsection5$]{Pi98} for details).
For a general system, Larsen-Pink has proved
the existence of 
a common $\Q$-form  of  $\bG_\ell\subset\GL_{V_\ell^{\ss}}$ for $\ell$ belonging to a set of primes of Dirichlet density $1$  if (\ref{1}) is absolutely irreducible and satisfies one of the following conditions \cite[Proposition 9.10]{LP92}:
\begin{enumerate}[($i$)]
\item the \emph{splitting field} of (\ref{1}) (see \cite[$\mathsection8.1$]{LP92}) is $\Q$;
\item the dimension of representations is divisible
neither by $3^{15}$ nor by the fifth power of an even integer strictly greater than $2$.
\end{enumerate}

\begin{definition}\label{newgp}
Let $\bG$ be a connected reductive group  defined over a field $F$ and $\Gamma$ a subgroup of $\bG(F)$. 
Denote by $\bG^{\ss}$ the quotient of $\bG$ by its radical and by $\Gamma^{\ss}$ the image of $\Gamma$ under the natural morphism
 $$\pi^{\ss}:\bG(F)\to\bG^{\ss}(F).$$ 
Denote by $\bG^{\der}$ the derived group of $\bG$, by $\bG^{\sc}$ the universal covering of $\bG^{\der}$, by $\pi^{\sc}$ the natural morphism
$$\pi^{\sc}:\bG^{\sc}(F)\to\bG^{\der}(F),$$
and by $\Gamma^{\sc}$ the pre-image of $\Gamma^{\ss}$ under $\pi^{\ss}\circ\pi^{\sc}$.
\end{definition}

The monodromy group $\Gamma_\ell$ is a compact $\ell$-adic Lie subgroup of $\bG_\ell(\Q_\ell)$. 
Identify $\bG_\ell$ as a connected reductive subgroup of $\GL_{k,\Q_\ell}$. The main results of this article
are as follows.

\begin{thm}\label{main}
Let $\{\Phi_\ell\}_\ell$ be the system (\ref{1}) and $\bG_\ell$ the connected algebraic monodromy group of $\Phi_\ell^{\ss}$ for all $\ell$.
Suppose Hypothesis \ref{A} is satisfied.
Then the following statements hold.
\begin{enumerate}[(i)]
\item The conjugacy class of $\bG_\ell\times_{\Q_\ell}\C$ in $\GL_{k,\C}$ is independent of $\ell$
\footnote{The reductive subgroups $\bG_{\ell_1}\times_{\Q_{\ell_1}}\C$ and $\bG_{\ell_2}\times_{\Q_{\ell_2}}\C$ are conjugate in
$\GL_{k,\C}$ for all distinct primes $\ell_1$ and $\ell_2$.}.
\item There exists 
a connected quasi-split reductive group $\bG_\Q$ defined over $\Q$ such that for all sufficiently 
large $\ell$,
\begin{equation*}
\bG_\ell\cong\bG_\Q\times_\Q\Q_\ell.
\end{equation*}
\end{enumerate}
\end{thm}

%\begin{definition}\label{image}
%Denote the Galois image $\Phi_\ell^{\ss}(\Gal_K)$ by $\Gamma_\ell$ for all $\ell$. Then $\Gamma_\ell$ is a subgroup of $\bG_\ell(\Q_\ell)$ for all $\ell$.
%\end{definition}

\begin{corollary}\label{newcor}
Let $\cG^{\sc}$ be
a semisimple group scheme  over $\Z[\frac{1}{N}]$ for some $N$ whose generic fiber is $\mathbf{G}_\Q^{\sc}$, where $\bG_\Q$ is in Theorem \ref{main}.
For all sufficiently large $\ell$, 
we have 
$$\Gamma_\ell^{\sc}\cong\cG^{\sc}(\Z_\ell).$$
\end{corollary}

\vspace{.1in}
Corollary \ref{newcor} can be applied to study the mod $\ell$ Galois images.
For any finite group $\bar\Gamma$, simple Lie type $\mathfrak{g}$ (e.g., $A_n$, $B_n$, $C_n$, $D_n$, $E_6$,...), and prime $\ell\geq5$, 
we defined in \cite{Hui15,HL16}(see $\mathsection$\ref{mod l}) the $\mathfrak{g}$-type $\ell$-rank $\rk_\ell^{\mathfrak{g}}\bar\Gamma$ of $\bar\Gamma$,
which measures the number of finite simple groups of type $\mathfrak{g}$ in characteristic $\ell$
in the composition series of $\bar\Gamma$. For example, 
\begin{equation*}
\rk^{\mathfrak{g}}_\ell\SL_{n+1}(\F_{\ell^f}) :=\left\{ \begin{array}{lll}
 fn &\mbox{if}\hspace{.1in} \mathfrak{g}=A_n,\\
 0 &\mbox{otherwise.}
\end{array}\right.
\end{equation*}

We studied the mod $\ell$ Galois image $\bar\Gamma_\ell:=\phi_\ell(\Gal_K)$ \emph{arising from \'etale cohomology} 
\footnote{Since $\Phi_\ell(\Gal_K)$ is compact, it fixes some $\Z_\ell$-lattice $L_\ell$ of $V_\ell$. Then $\phi_\ell$ 
is defined to be the semisimplification of the mod $\ell$ reduction of $\Phi_\ell$ with respect to $L_\ell$.}
for all sufficiently large $\ell$ in \cite{Hui15} and showed 
that $\rk_\ell^{A_n}\bar\Gamma_\ell$
%, the \emph{$A_n$-type $\ell$-rank} of $\bar\Gamma_\ell$ 
is independent of $\ell\gg0$ if 
$n\in\N\backslash\{1,2,3,4,5,7,8\}$ (see $\mathsection$\ref{mod l}). However, the function $\rk_\ell^{A_n}$ cannot distinguish 
between the \emph{Chevalley group} $A_n(\ell^f)$ and the \emph{Steinberg group} ${}^2\!A_n(\ell^{2f})$ for $n\geq 2$ since their $A_n$-type $\ell$-ranks are both $fn$.
For example, suppose 
$A_6$ is the only simple factor of $\mathfrak{g}_{\ell_0}$,
then $\bar\Gamma_\ell$ has only one composition factor of Lie type in characteristic $\ell$
for $\ell\gg0$, which is
either the Chevalley group $A_6(\ell)$ or the Steinberg group ${}^2\!A_6(\ell^{2})$. 
One cannot tell which one occurs for large $\ell$ from the results in \cite{Hui15}.
Nevertheless,
Corollary \ref{cor} below provides a precise description of 
the composition factors of Lie type in characteristic $\ell$ of $\bar\Gamma_\ell$
for $\ell\gg0$ if Hypothesis \ref{A} is satisfied.

\begin{definition}\label{Lie}
For any prime $\ell\geq5$ and finite group $\bar\Gamma$, denote by $\mathrm{Lie}_\ell\bar\Gamma$ the multiset of 
the composition factors
of Lie type in characteristic $\ell$ of $\bar\Gamma$.
\end{definition}

\begin{cor}\label{cor}
Let $\cG^{\der}$ be
a semisimple group scheme  over $\Z[\frac{1}{N}]$ for some $N$ whose generic fiber is $\mathbf{G}_\Q^{\der}$, 
where $\bG_\Q$ is in Theorem \ref{main}.
For all sufficiently large $\ell$, 
we have 
$$\mathrm{Lie}_\ell\bar\Gamma_\ell=\mathrm{Lie}_\ell\cG^{\der}(\F_\ell).$$
\end{cor}

\begin{remark} For the $A_6$ case discussed above, Corollary \ref{cor} implies (by studying the $\Gal_\Q$ action on the Dynkin diagram of  $\mathbf{G}_\Q^{\der}$) either the Chevalley group $A_6(\ell)$ occurs for $\ell\gg0$ or there 
is a quadratic extension $F$ of $\Q$ such that for $\ell\gg0$, the Chevalley group $A_6(\ell)$ occurs for $\ell$ that splits completely and 
the Steinberg group ${}^2\!A_6(\ell^{2})$ occurs for $\ell$ that is inert. Such a congruence is useful to the inverse Galois problem and appears, for example, 
in the computation of the geometric $\Z/\ell\Z$-monodromy of the moduli space of trielliptic curves \cite[Theorem 3.8]{AP07}.\end{remark}

Let us sketch the proof of Theorem \ref{main}. 
For any connected reductive subgroup $\bG$ of $\GL_{k,F}$, we introduce the notion of the
\emph{formal bi-characters} of $\bG$ in Definition \ref{formal2}.
Since $\bG_\ell$ is a connected reductive subgroup of $\GL_{k,\Q_\ell}$ for all $\ell$,
the method in \cite[$\mathsection3$]{Hui13} shows that the isomorphism class of the formal bi-characters of $\bG_\ell\times_{\Q_\ell}\C\subset\GL_{k,\C}$ (for any embedding $\Q_\ell\hookrightarrow\C$) is independent of $\ell$ (Theorem \ref{bicharC}). 
By the method of \emph{Serre's Frobenius tori} ($\mathsection\ref{FT}$),
one can pick for each large $\ell$ a formal bi-character of $\bG_\ell$ such that these formal bi-characters admit
a common $\Q$-form up to conjugation (Theorem \ref{bicharQ}).
Under Hypothesis \ref{A}, the invariance of both the formal bi-characters of $\bG_\ell\times_{\Q_\ell}\C$ 
and \emph{the positions of roots in the weight space} ($\mathsection\ref{s3.1}$)
imply by Theorem \ref{rdr} that: 
\begin{enumerate}[($i$)]
\item the root datum of $\bG_\ell\times_{\Q_\ell}\C$ is independent of $\ell$;
\item the conjugacy class of $\bG_\ell\times_{\Q_\ell}\C$ in $\GL_{k,\C}$ is independent of $\ell$.
\end{enumerate}
The assertion (ii) above is exactly Theorem \ref{main}(i).
We also know that $\bG_\ell$ is quasi-split for $\ell\gg0$ by Hypothesis \ref{A} and Corollary \ref{maxcor}.
The techniques on forms of reductive groups that are essential to the proof of Theorem \ref{main}(ii)
are reviewed in $\mathsection\ref{s4}$.
By exploiting these techniques and all the $\ell$-independence results above, we prove the existence of a common $\Q$-form $\bG_\Q$ for $\{\bG_\ell\}_{\ell\gg0}$ 
in $\mathsection\ref{s5}$, which completes Theorem \ref{main}(ii).
  
\begin{remark} 
The $\ell$-independence theorem of this paper and the results on the invariance of roots in $\mathsection3$ allow us to 
study the decomposition of compatible system of Galois representations arising from geometry into a direct sum of an
abelian and a non-abelian compatible 
subsystems 
\cite{Hui16}.\end{remark}

\section{Some results on $\ell$-adic representations}

\subsection{Strictly compatible systems} 

Let $k$ be a positive integer, $K$ a number field, and $\bar K$ an algebraic closure of $K$.
Denote by $\Gal_K$ the absolute Galois group of $K$ and by 
$\Sigma_K$ (resp. $\Sigma_{\bar K}$) 
the set of non-Archimedean valuations of $K$ (resp. $\bar K$). For each prime number $\ell$, let $\Psi_\ell$
be a $k$-dimensional, continuous $\ell$-adic representation of $K$,
$$\Psi_\ell:\Gal_K\to \GL_k(\Q_\ell).$$
For $v\in\Sigma_K$, let $\bar v\in\Sigma_{\bar K}$ divide $v$. Denote by $D_{\bar v}$ and $I_{\bar v}$ the decomposition subgroup and inertia subgroup of 
$\Gal_K$ at $\bar v$ respectively. Let $k(v)$ be the residue field of $K$ completed with respect to $v$.
Since $D_{\bar v}/I_{\bar v}\cong\Gal_{k(v)}$ naturally, denote by $\mathrm{Frob}_{\bar v}\in D_{\bar v}/I_{\bar v}$ the element corresponding to the inverse of the Frobenius automorphism of $\overline{k(v)}/k(v)$ and call it a
\emph{Frobenius element}.
Suppose $\bar v$ and $\bar v'$ both divide $v\in\Sigma_K$. Then the two pairs $I_{\bar v}\subset D_{\bar v}$ and
$I_{\bar v'}\subset D_{\bar v'}$ of closed subgroups are conjugate in $\Gal_K$.
The representation $\Psi_\ell$ is said to be \emph{unramified} at $v$ if $\Psi_\ell(I_{\bar v})$ is trivial for some $\bar v$ dividing $v$.
In this case, it makes sense to define the image of Frobenius element $\Psi_\ell(\mathrm{Frob}_{\bar v})$.

\begin{definition}\label{comsys} The system of $\ell$-adic representations $\{\Psi_\ell\}_\ell$
is said to be \textit{strictly compatible} if the following conditions are satisfied.
\begin{enumerate}
\item[($i$)] There is a finite subset $S\subset\Sigma_K$ such that $\Psi_\ell$ is unramified outside $S_\ell:=S\cup\{v\in\Sigma_K: v|\ell \}$ for all $\ell$.
\item[($ii$)] For all primes $\ell_1\neq\ell_2$ and $\bar{v}\in\Sigma_{\bar{K}}$ dividing $v\in\Sigma_K\backslash(S_{\ell_1}\cup S_{\ell_2})$, the characteristic polynomials of $\Psi_{\ell_1}(\mathrm{Frob}_{\bar{v}})$ and $\Psi_{\ell_2}(\mathrm{Frob}_{\bar{v}})$ are equal
to some polynomial $P_v(x)\in\mathbb{Q}[x]$ depending only on $v$.
\end{enumerate}
\end{definition}

\noindent\textbf{Examples of strictly compatible systems}. 
\begin{enumerate}[($i$)]
\item The semisimplification $\{\Psi_\ell^{\ss}\}_\ell$ of the strictly compatible system $\{\Psi_\ell\}_\ell$. Note that the characteristic polynomials of $\Psi_{\ell}(\mathrm{Frob}_{\bar{v}})$ and $\Psi_{\ell}^{\ss}(\mathrm{Frob}_{\bar{v}})$ are equal.
\item The direct sum of two strictly compatible systems.
\item The system of abelian $\ell$-adic representations arising from a $\Q$-representation 
of the \emph{Serre group} $\bS_\mathfrak{m}$ \cite{Se98}. 
\item The system of $\ell$-adic representations arising from the $\ell$-adic Tate modules of an abelian variety $A$ defined over $K$.
\item The system of $\ell$-adic representations arising from \'etale cohomology as in (\ref{1}). 
\end{enumerate}

\subsection{Formal character and bi-character}  \label{FCBC}

Let $F$ be a field and $\bG$ a connected reductive subgroup of $\GL_{k,F}$.
Since $\bG$ is connected, the derived subgroup $\bG^{\der}$ is semisimple.

\begin{definition}\label{formal1}
Let $\bT$ be a maximal torus of $\bG$. Then the natural inclusion $\bT\subset\GL_{k,F}$ is said to be 
a \emph{formal character} of $\bG\subset\GL_{k,F}$ or of $\bG$ for simplicity.
Two formal characters $\bT_1\subset\GL_{k,F}$ and $\bT_2\subset\GL_{k,F}$ of respectively $\bG_1$ and $\bG_2$  are isomorphic if $\bT_1$ and $\bT_2$ are conjugate in $\GL_{k,F}$, i.e., conjugate by an element of $\GL_k(F)$.
\end{definition}

\begin{definition}\label{formal2}
Let $\bT$ be a maximal torus of $\bG$ and $\bT^{\ss}:=(\bT\cap \bG^{\der})^\circ$ a maximal torus of $\bG^{\der}$. Then the chain $\bT^{\ss}\subset\bT\subset\GL_{k,F}$ is said to be 
a \emph{formal bi-character} of $\bG\subset\GL_{k,F}$ or of $\bG$ for simplicity.
Two formal bi-characters $\bT_1^{\ss}\subset\bT_1\subset\GL_{k,F}$ and $\bT_2^{\ss}\subset\bT_2\subset\GL_{k,F}$ of respectively $\bG_1$ and $\bG_2$ are isomorphic if the two pairs $\bT_1^{\ss}\subset \bT_1$ and $\bT_2^{\ss}\subset\bT_2$ are conjugate in $\GL_{k,F}$, 
 i.e., conjugate by an element of $\GL_k(F)$.
\end{definition}

\begin{remark}\label{formalrem}
If $F$ is algebraically closed, then all formal characters (formal bi-characters) of $\bG\subset\GL_{k,F}$ are isomorphic since all maximal tori of $\bG$ are conjugate in $\bG$.
\end{remark}

\subsection{Frobenius tori}\label{FT}

Let $\{\Psi_\ell\}_\ell$ be a semisimple, $k$-dimensional, strictly compatible system of $\ell$-adic representations.
Denote by $\bG_\ell$ the algebraic monodromy group at $\ell$, i.e., the Zariski closure of $\Psi_\ell(\Gal_K)$ in $\GL_{k,\Q_\ell}$. Assume 
 $\bG_\ell$ is a connected reductive subgroup of $\GL_{k,\Q_\ell}$ for all $\ell$.
Since $\Psi_\ell$ is unramified outside $S_\ell$ (Definition \ref{comsys}), the image of the set of \emph{Frobenius elements} 
$$\mathscr{F}_\ell:=\{\Psi_\ell(\mathrm{Frob}_{\bar v}): \bar v~\mathrm{divides}~v\notin S_\ell\}$$
is dense in the Galois image $\Psi_\ell(\Gal_K)$ by the Cheboterav density theorem. 
It is also Zariski dense in $\bG_\ell$ by the definition of $\bG_\ell$.
Definition \ref{Frob}, Theorem \ref{FMT}, and its corollaries below are due to Serre \cite{Se81}.  

\begin{definition}\label{Frob}
For each $\bar v$ dividing $v\notin S_\ell$, the \emph{Frobenius torus} $\bT_{\bar v,\ell}$ is defined as the identity component 
of the smallest algebraic subgroup of $\bG_\ell$ containing the semisimple part of $\Psi_\ell(\mathrm{Frob}_{\bar v})$.
\end{definition}

The following theorem uses the terminology of Larsen-Pink \cite[Theorem 1.2]{LP97}, see also \cite[Theorem 3.7]{Chi92}.
 
\begin{thm}\label{FMT}(Serre)
Let $\ell$ be a prime and $v\in\Sigma_K$. 
Denote the characteristic polynomial of $\Psi_\ell(\mathrm{Frob}_{\bar v})\in \mathscr{F}_\ell$ by $P_v(x)\in\Q[x]$, 
which is independent of $\ell$. 
Denote by $p_v$ the characteristic of $v$ and by 
$q_v$ the cardinality of the residue field of $v$. 
Suppose the following conditions are satisfied for the roots $\alpha$ of $P_v(x)$:
\begin{enumerate}[(a)]
\item the absolute values of $\alpha$ in all complex embeddings are equal;
\item $\alpha$ is a unit at any non-Archimedean place not above $p_v$;
\item for any non-Archimedean valuation $w$ of $\bar \Q$ such that $w(p_v)>0$, the ratio $w(\alpha)/w(q_v)$
belongs to a finite subset of $\Q$ that is independent of $\bar v$.
\end{enumerate}
Then there exists a proper closed subvariety $\bY$ of $\bG_\ell$ such that $\bT_{\bar v,\ell}$
is a maximal torus of $\bG_\ell$ whenever $\Psi_\ell(\mathrm{Frob}_{\bar v})\in\bG_\ell\backslash \bY$.
\end{thm}

Since the Frobenius tori $\bT_{\bar v,\ell}$ and $\bT_{\bar v',\ell}$ are conjugate whenever $\bar v|_K=v=\bar v'|_K$, the following corollary follows directly.

\begin{cor}\label{FMTC1} \cite[Corollary 3.8]{Chi92}, \cite[Corollary 1.4]{LP97}
The following subset of $\Sigma_K$ is of Dirichlet density $1$,
$$\{v\in\Sigma_K\backslash S_\ell: \bT_{\bar v,\ell}~\mathrm{is~a~maximal~torus~of~} \bG_\ell\}.$$
\end{cor}

If we embed $\Q_\ell$ in $\C$, then $\bG_\ell\times_{\Q_\ell}\C$ is a connected $\C$-reductive subgroup of $\GL_{k,\C}$ for all $\ell$.

\begin{cor}\label{FMTC2}
The isomorphism class of the formal characters of $\bG_\ell\times_{\Q_\ell}\C\subset\GL_{k,\C}$  is independent of $\ell$. In particular, the rank of $\bG_\ell$ is independent of $\ell$.
\end{cor}

\begin{proof}
For all distinct primes $\ell$ and $\ell'$, there exists $\bar v$ 
such that $\bT_{\bar v,\ell}\times_{\Q_\ell}\C$ and $\bT_{\bar v,\ell'}\times_{\Q_{\ell'}}\C$ 
are  maximal tori of $\bG_\ell\times_{\Q_\ell}\C$ and $\bG_{\ell'}\times_{\Q_{\ell'}}\C$ respectively
by Corollary \ref{FMTC1}. Since $\bT_{\bar v,\ell}\times_{\Q_\ell}\C$ and $\bT_{\bar v,\ell'}\times_{\Q_{\ell'}}\C$
only depend on the eigenvalues of $P_v(x)$ (Definition \ref{comsys}(ii)), they are conjugate
in $\GL_{k,\C}$. Therefore, the first assertion of the corollary holds
by Remark \ref{formalrem}. Since the rank of $\bG_\ell$ is defined as the dimension of 
a maximal torus, it is independent of $\ell$.
\end{proof}

\begin{cor}\label{FMTC3}
There exist a $\Q$-subtorus $\bT_\Q$ of $\GL_{k,\Q}$ and for every sufficiently large $\ell$, a formal character $\bT_\ell\subset\GL_{k,\Q_\ell}$ of $\bG_\ell$ such that $\bT_\Q\subset\GL_{k,\Q}$ is a common $\Q$-form of $\{\bT_\ell\subset\GL_{k,\Q_\ell}\}_{\ell\gg0}$ up to conjugation, i.e., the subtori $\bT_\Q\times_{\Q}\Q_\ell$ and $\bT_\ell$ are conjugate by an element of $\GL_k(\Q_\ell)$ if $\ell$ is sufficiently large.
\end{cor}

\begin{proof}
Let $\ell'$ be a prime. Then there exist $v\notin S$ (Definition \ref{comsys}(i)) and $\bar v$ dividing $v$ such that $\bT_{\ell'}:=\bT_{\bar v,\ell'}$ is a maximal torus of $\bG_{\ell'}$ by Corollary \ref{FMTC1}. For $\ell\gg0$, $\bT_{\ell}:=\bT_{\bar v,\ell}$ is a subtorus of $\bG_\ell$ by construction. Since the rank of $\bG_\ell$ is independent of $\ell$ by Corollary \ref{FMTC2}, $\bT_\ell$ is a maximal torus of $\bG_\ell$ for all $\ell\gg0$.
Let $A_v\in\GL_k(\Q)$ be a semisimple matrix with characteristic polynomial $P_v(x)$ (Definition \ref{comsys}(ii)). Then $A_v$ is conjugate in $\GL_k(\Q_\ell)$ to the semisimple part of $\Psi_\ell(\mathrm{Frob}_{\bar v})$ for all $\ell\gg0$. Hence, if we denote by $\bT_\Q$ the identity component of the smallest algebraic subgroup of $\GL_{k,\Q}$ containing $A_v$, then $\bT_\Q\subset\GL_{k,\Q}$ is a common $\Q$-form of $\{\bT_\ell\subset\GL_{k,\Q_\ell}\}_{\ell\gg0}$ up to conjugation.
\end{proof}

\subsection{$\ell$-independence of the $\ell$-adic images}\label{l-adic}

We follow the terminology in $\mathsection\ref{FT}$.
Let $i:\bS_\mathfrak{m}\to\GL_{m,\Q}$ be a representation of some Serre group $\bS_\mathfrak{m}$ of number field $K$. Then attached to this morphism is a strictly compatible system of abelian semisimple $\ell$-adic representations $\{\Theta_\ell\}_\ell$ of $K$ \cite[$\mathsection2.2$]{Se98}. Suppose $i$ is injective. Consider the direct sum of two strictly compatible systems, 
\begin{equation}\label{sumsys}
\{\Psi_\ell\oplus\Theta_\ell:\Gal_K\to \GL_k(\Q_\ell)\times\GL_m(\Q_\ell)\subset\GL_{k+m}(\Q_\ell)\}_\ell.
\end{equation}
Define $p_1:\GL_{k}\times\GL_{m}\to \GL_{k}$ and $p_2:\GL_{k}\times\GL_{m}\to \GL_{m}$ 
to be the projection to the first and the second factors respectively. 
Let $\bG_\ell'\subset\GL_{k,\Q_\ell}\times\GL_{m,\Q_\ell}$ be the algebraic monodromy group at $\ell$, which 
is assumed to be connected.
Let $\bT_\ell'$ be a maximal torus of $\bG_\ell'$. Then $p_1(\bT_\ell')$ is a maximal torus of $\bG_\ell$, the algebraic monodromy group of $\Psi_\ell$.
We showed in \cite[$\mathsection3$]{Hui13} that the conjugacy class of the subtorus
$$\bT_\ell'\times_{\Q_{\ell}}\C\subset\GL_{k,\C}\times\GL_{m,\C}$$ 
is independent of $\ell$,
i.e., for all primes $\ell$ and $\ell'$, the subtori 
$\bT_\ell'\times_{\Q_\ell}\C$ and $\bT_{\ell'}'\times_{\Q_{\ell'}}\C$ are conjugate in $\GL_{k,\C}\times\GL_{m,\C}$.
Define
\begin{align}
\begin{split}
\bT^{\ss}_\C:&=p_1((\mathrm{Ker}(p_2)\cap\bT_\ell')^\circ)\times_{\Q_\ell}\C;\\
\bT_\C:&=p_1(\bT_\ell')\times_{\Q_\ell}\C.
\end{split}
\end{align}
It follows that the conjugacy class of the chain of subtori
$\bT^{\ss}_\C\subset \bT_\C\subset\GL_{k,\C}$ is independent of $\ell$.

\begin{thm}\cite[Theorem 3.19]{Hui13}\label{bicharC}
The complex torus $\bT^{\ss}_\C$ is a maximal torus of $\bG_\ell^{\der}\times_{\Q_\ell}\C$ 
and the isomorphism class of the formal bi-character
\begin{equation*}\label{formbi}
\bT^{\ss}_\C\subset \bT_\C\subset\GL_{k,\C}
\end{equation*}
 of $\bG_\ell\times_{\Q_\ell}\C$ 
is independent of $\ell$. In particular, the semisimple rank of $\bG_\ell$ is independent of $\ell$.
\end{thm}

Let $\Psi_\ell$ be $\Phi_\ell^{\ss}$ for all $\ell$.
By combining all the results in this subsection, we obtain the following theorem for the system (\ref{1}).

\begin{thm}\label{bicharQ}
Let $\{\Phi_\ell\}_\ell$ be the system (\ref{1}) and $\bG_\ell$ the connected algebraic monodromy group of $\Phi_\ell^{\ss}$ for all $\ell$.
There exist two $\Q$-subtori $\bT^{\ss}_\Q\subset\bT_\Q$ of $\GL_{k,\Q}$ and a formal bi-character $\bT_\ell^{\ss}\subset\bT_\ell\subset\GL_{k,\Q_\ell}$ of $\bG_\ell$ for all sufficiently large $\ell$ such that $\bT^{\ss}_\Q\subset\bT_\Q\subset\GL_{k,\Q}$ is a common $\Q$-form of $\{\bT_\ell^{\ss}\subset\bT_\ell\subset\GL_{k,\Q_\ell}\}_{\ell\gg0}$ up to conjugation.
\end{thm}

\begin{proof}
Since $\Phi_\ell^{\ss}$ and $\Theta_\ell$ satisfy the conditions ($a$), ($b$), ($c$) of Theorem \ref{FMT} (\cite[Theorem 1.1]{LP97}, \cite[Chapter 2 $\mathsection3.4$]{Se97}), so does $\Phi_\ell^{\ss}\oplus\Theta_\ell$.
Since $\{\Phi_\ell^{\ss}\}_\ell$ and $\{\Theta_\ell\}_\ell$ are both strictly compatible,
 there exists a
formal character 
$$\bT_\ell'\subset\GL_{k,\Q_\ell}\times\GL_{m,\Q_\ell}\subset\GL_{k+m,\Q_\ell}$$
 of $\bG_\ell'$ 
such that these formal characters have a common
$\Q$-form up to conjugation in $\GL_k\times\GL_m$
$$\bT'_\Q\subset\GL_{k,\Q}\times\GL_{m,\Q}\subset\GL_{k+m,\Q}$$
for all sufficiently large $\ell$
 by Corollary \ref{FMTC3}.
Define two $\Q$-tori $\bT^{\ss}_\Q:=p_1((\mathrm{Ker}(p_2)\cap\bT'_\Q)^\circ)$ and $\bT_\Q:=p_1(\bT'_\Q)$. 
Define two $\Q_\ell$-tori $\bT_\ell^{\ss}:=p_1((\mathrm{Ker}(p_2)\cap\bT_\ell')^\circ)$ and $\bT_\ell:=p_1(\bT_\ell')$. Then 
\begin{equation*}
\bT_\ell^{\ss}\subset \bT_\ell\subset\GL_{k,\Q_\ell}
\end{equation*}
is a formal bi-character of $\bG_\ell$ by Theorem \ref{bicharC} and admits a
$\Q$-form $\bT^{\ss}_\Q\subset\bT_\Q\subset\GL_{k,\Q}$ by construction if $\ell$ is sufficiently large.
\end{proof}

Let $\mathfrak g_\ell^{\der}$ be the Lie algebra of $\bG_\ell^{\der}\times_{\Q_\ell}\C$. 
Since the isomorphism class of the formal characters of $\bG_\ell^{\der}\times_{\Q_\ell}\C\subset\GL_{k,\C}$ 
is independent of $\ell$ (Theorem \ref{bicharC}), 
the isomorphism class of the formal characters of $\mathfrak g_\ell^{\der}\subset\End_k(\C)$ (in the sense of \cite[$\mathsection2.1$]{Hui13})
 is likewise independent of $\ell$.
We obtained 
the following $\ell$-independence result by studying the positions of roots in the weight space
 \cite[$\mathsection2$]{Hui13}. Relevant details will be given in $\mathsection\ref{s3.1}$.

\begin{thm}\label{adic}\cite[Theorem 3.21]{Hui13}
Let $\mathfrak{g}_\ell$ be the Lie algebra of $\bG_\ell\times_{\Q_\ell}\C$ and $a_{n,\ell}$
 the number of $A_n$ factors of $\mathfrak{g}_\ell$. Then the following statements hold.
\begin{enumerate}[($i$)]
\item The parity of $a_{4,\ell}$ is independent of $\ell$;
\item The number $a_{n,\ell}$ is independent of $\ell$ 
if $n\in\N\backslash\{1,2,3,4,5,7,8\}$.
\end{enumerate}
\end{thm}

\begin{cor}\label{adiccor}
Suppose Hypothesis \ref{A} holds, then the complex reductive Lie algebra
$\mathfrak{g}_\ell$ is independent of $\ell$.
\end{cor}

\begin{proof}
By Corollary \ref{FMTC2} and Theorem \ref{bicharC}, the semisimple rank and the dimension of the center of $\mathfrak{g}_\ell$ are both independent of $\ell$.
The corollary follows from Theorem \ref{adic}.
\end{proof}

\subsection{$\ell$-independence of the mod $\ell$ images}\label{mod l}

Let $\ell \ge 5$ be a prime and $\LT$ a \emph{Lie type} (e.g., $A_n,B_n,C_n,D_n,...$). 
We define the \emph{$\LT$-type $\ell$-rank} function, $\rk^{\LT}_\ell$, and 
the \emph{total $\ell$-rank} function, $\rk_\ell$, on
finite groups. 
The dimension of an algebraic group $\bG /F$ as an $F$-variety is denoted by $\dim\bG$.
Let $\bar\Gamma$ be a finite simple group of Lie type in characteristic $\ell$. 
Then there exists an adjoint simple group $\bar{\bG}/\F_{\ell^f}$ such that
$$\bar\Gamma=\bar{\bG}(\F_{\ell^f})^{\der},$$
the derived group of the group of $\F_{\ell^f}$-rational points of $\bar{\bG}$.
By base change to $\bar\F_\ell$, we obtain
$$\bar{\bG}\times_{\F_{\ell^f}}\bar\F_\ell=\bar{\bH}^m,$$
where $\bar{\bH}$
is an $\bar\F_\ell$-adjoint simple group of some Lie type $\mathfrak{h}$.
We then set the $\LT$-type $\ell$-rank of $\bar\Gamma$ to be
\begin{equation*}
\rk^{\LT}_\ell\bar\Gamma :=\left\{ \begin{array}{lll}
 f\cdot\rk \bar{\bG} &\mbox{if}\hspace{.1in} \LT=\mathfrak{h},\\
 0 &\mbox{otherwise,}
\end{array}\right.
\end{equation*}
and the total $\ell$-rank of $\bar\Gamma$ to be 
$$\rk_\ell\bar\Gamma:=\sum_{\LT}\rk_\ell^{\LT}\bar\Gamma.$$
 %(i.e. $\sigma$ permutes the simple components of $\bG $ over $\bar\F_\ell$ in a single orbit).
For simple groups which are not of Lie type in characteristic $\ell$, 
we define the $\LT$-type $\ell$-rank and the total $\ell$-rank to be zero.  
We extend the definitions
to arbitrary finite groups by defining the $\LT$-type $\ell$-rank and the total $\ell$-rank of any finite group to be the sum of the ranks of its composition factors.
This definition makes it clear that $\rk_\ell^{\LT}$ and $\rk_\ell$ are additive on short exact sequences of groups.  In particular, the $\LT$-type $\ell$-rank and the total $\ell$-rank of
every solvable finite group are zero.

Given a strictly compatible system $\{\Psi_\ell\}_\ell$, the monodromy group $\Psi_\ell(\Gal_K)$ is a compact subgroup of $\GL_k(\Q_\ell)$ which fixes some $\Z_\ell$-lattice of $\Q_\ell^k$ for all $\ell$. By some change of coordinates, we obtain for each $\ell$ a unique semisimple mod $\ell$ representation
$$\psi_\ell:\Gal_K\to\GL_k(\F_\ell)$$ 
by reduction mod $\ell$ and semisimplification (Brauer-Nesbitt \cite[Theorem 30.16]{CR88}). We then say that
 the mod $\ell$ system 
$\{\psi_\ell\}_\ell$ \emph{arises from the $\ell$-adic system} $\{\Psi_\ell\}_\ell$.
The following theorem is the mod $\ell$ analogue of (part of) Theorem \ref{bicharC} and Theorem \ref{adic}.

\begin{thm}\label{mod}\cite[Theorem A, Corollary B]{Hui15} 
Let  
 $\{\phi_\ell\}_{\ell\in\mathscr{P}}$ be the system of  mod $\ell$ representations arising from 
the system $\{\Phi_\ell^{\ss}\}_\ell$ 
and $\bG_\ell$ the 
 connected reductive algebraic monodromy group of $\Phi_\ell^{\ss}$ for all $\ell$. Denote the image of $\phi_\ell$ by $\bar\Gamma_\ell$, then the following statements hold for $\ell\gg0$.
\begin{enumerate}
\item[($i$)] The total $\ell$-rank $\rk_\ell\bar\Gamma_\ell$ of $\bar\Gamma_\ell$ is equal to the rank of $\bG_\ell^{\der}$ and is therefore independent of $\ell$.
\item[($ii$)] The $A_n$-type $\ell$-rank $\rk_\ell^{A_n}\bar\Gamma_\ell$ of $\bar\Gamma_\ell$ for $n\in\N\backslash\{1,2,3,4,5,7,8\}$ and 
the parity of $(\rk_\ell^{A_4}\bar\Gamma_\ell)/4$ are independent of $\ell$. 
\end{enumerate}
\end{thm}

\subsection{Maximality of the $\ell$-adic images}\label{2.6}

Let $\{\Phi_\ell\}_\ell$ be the system (\ref{1}).
The representation $\Phi_\ell^{\ss}$ and the algebraic monodromy group $\bG_\ell$ 
are said to be of type A if every simple factor of 
$\mathfrak{g}_\ell:=\mathrm{Lie}(\bG_\ell\times_{\Q_\ell}\C)$ is of type $A_n$.
The maximality of the monodromy group $\Gamma_\ell$ inside the 
$\ell$-adic Lie group $\bG_\ell(\Q_\ell)$ is studied in \cite{HL16} assuming that $\bG_\ell$ is of type A. 
A reductive group $\bH/\Q_\ell$ is said to be \emph{unramified} if
it is quasi-split over $\Q_\ell$ and splits over an unramified extension of $\Q_\ell$.

\begin{thm}\label{max}\cite[Main theorem]{HL16}
Let $\{\Phi_\ell\}_\ell$ be the system (\ref{1}). For all sufficiently large $\ell$, if  $\bG_\ell$ is of type A, then $\Gamma_\ell^{\sc}$ is 
a hyperspecial maximal compact subgroup of $\bG_\ell^{\sc}(\Q_\ell)$ and $\bG_\ell^{\sc}$ is unramified over $\Q_\ell$.
\end{thm}

\begin{cor}\label{maxcor}
For all sufficiently large $\ell$, if $\bG_\ell$ is of type A, then $\bG_\ell$ is unramified.
\end{cor}

\begin{proof}
For all sufficiently large $\ell$, $\bG_\ell^{\sc}$ is unramified over $\Q_\ell$  by Theorem \ref{max} and $\bG_\ell$ splits over an unramified extension by Corollary \ref{FMTC3}. Since the natural morphism $\bG_\ell^{\sc}\to\bG_\ell^{\der}$ is a $\Q_\ell$-isogeny and the center of $\bG_\ell$ is defined over $\Q_\ell$, $\bG_\ell$ is unramified for $\ell\gg0$.
\end{proof}

\begin{remark}
If $X$ is an abelian variety, then the conclusions of Theorem \ref{max} hold without any type A assumption on $\bG_\ell$ (forthcoming).
\end{remark}

\section{$\ell$-independence of $\bG_\ell\subset\GL_{k,\Q_\ell}$}\label{s3}
Let $\{\Phi_\ell\}_\ell$ be the system (\ref{1}). The algebraic  monodromy group $\bG_\ell$ of $\Phi_\ell^{\ss}$
is a reductive subgroup of $\GL_{k,\Q_\ell}$. We suppose $\bG_\ell$ is connected for all $\ell$ by requiring $K=K^{\mathrm{conn}}$. 
We embed $\Q_\ell$ in $\C$ for all $\ell$, then 
$\bG_\ell\times_{\Q_\ell}\C$ is a subgroup of $\GL_{k,\C}$ for all $\ell$
and the isomorphism class of the formal bi-characters of $\bG_\ell\times_{\Q_\ell}\C$ is independent of $\ell$ by Theorem \ref{bicharC}.
If $\bG_\ell\times_{\Q_\ell}\C$ is semisimple and the tautological representation on $\C^k$ is irreducible for all $\ell$, 
then a formal character 
determines the root lattice and the set of short roots of $\bG_\ell\times_{\Q_\ell}\C$ \cite[$\mathsection4$ Proposition]{LP90}.
In a lot of cases, 
the above information determines the root system of $\bG_\ell\times_{\Q_\ell}\C$ and the representation $\bG_\ell\times_{\Q_\ell}\C\subset\GL_{k,\C}$ \cite[Theorem 4]{LP90}, which implies that the
conjugacy class of $\bG_\ell\times_{\Q_\ell}\C$ in $\GL_{k,\C}$ is independent of $\ell$.
The purpose of this section is to prove that if Hypothesis \ref{A} holds, then
the formal bi-character 
$$\bT^{\ss}_\C\subset\bT_\C\subset\GL_{k,\C}$$
of $\bG_\ell\times_{\Q_\ell}\C$ (Theorem \ref{bicharC}) determines the root datum \cite[$\mathsection1$]{Sp79} 
of $(\bG_\ell\times_{\Q_\ell}\C,\bT_\C)$
and the conjugacy class of $\bG_\ell\times_{\Q_\ell}\C$ (in $\GL_{k,\C}$) for all $\ell$ (Theorem \ref{rdr}).
All these are based on the crucial root computations in \cite[$\mathsection2$]{Hui13}, which will be explained below.

\subsection{The invariance of the roots in the weight space}\label{s3.1}

Let $\mathfrak g$ and $\mathfrak g'$ be two complex semisimple subalgebras of $\End_k(\C)$. 
Suppose $\mathfrak t\subset \End_k(\C)$ is a common Cartan subalgebra of $\mathfrak g$ and $\mathfrak g'$. 
The following notations are defined with respect to $\mathfrak t$. 
Let $R$ and $W$ (resp. $R'$ and $W'$) be the roots and Weyl group of $\mathfrak g$ (resp. $\mathfrak g'$) respectively.
The semisimple Lie algebras $\mathfrak g$ and $\mathfrak g'$  have the same weight lattice $\Lambda\subset\mathfrak t^*$,
generated by the weights $\{\alpha_1,...,\alpha_k\}$ of the faithful representation $\mathfrak t\subset \End_k(\C)$.
Therefore, we say the
faithful representations $\mathfrak g\subset \End_k(\C)$ and $\mathfrak g'\subset\End_k(\C)$ have identical formal character (\cite[$\mathsection2.1$]{Hui13}) and we define
\begin{equation*}\label{fcLie}
\mathrm{Char}(\C^k):=\alpha_1+\alpha_2+\cdots+\alpha_k\in\Z[\Lambda].
\end{equation*}
Since the weights in $\mathrm{Char}(\C^k)$ generate the \emph{weight space} $\Lambda\otimes_\Z\R$, one can define a positive definite
inner product $((~,~))$ on $\Lambda\otimes_\Z\R$ that is isomorphic to the $\R$-span of $\Lambda$ in $\mathfrak t^*$ such that 
the finite subgroup of $\GL(\Lambda\otimes_\Z\R)$ preserving $\mathrm{Char}(\C^k)$ is orthogonal \cite[$\mathsection2.3$]{Hui13}. 
Let $\{\mathfrak q_i\}_i$ and $\{\mathfrak q_j'\}_j$ be the multiset of simple factors of $\mathfrak g$ and $\mathfrak g'$ respectively.
Denote by $R_i$, $\Lambda_i$, and $\Lambda_i\otimes_\Z\R$ (resp. $R_j'$, $\Lambda_j'$, and $\Lambda_j'\otimes_\Z\R$) the roots, 
the weight lattice, and the
weight space of the simple Lie algebra $\mathfrak q_i$ (resp. $\mathfrak q_j'$) with respect to $\mathfrak t\cap\mathfrak q_i$ (resp. $\mathfrak t\cap\mathfrak q_j'$)  respectively. 
Then $\Lambda_i\otimes_\Z\R$ (resp. $\Lambda_j'\otimes_\Z\R$) 
can be identified as a subspace of $\Lambda\otimes_\Z\R$. 
We obtain $R=\bigcup_i R_i$ (resp. $R'=\bigcup_j R_j'$).

\begin{lemma}\label{Killing}
\begin{enumerate}[($i$)]
\item The weight subspaces $\Lambda_{i_1}\otimes_\Z\R$ and $\Lambda_{i_2}\otimes_\Z\R$ of $\Lambda\otimes_\Z\R$ are orthogonal with respect to $((~,~))$ whenever $i_1\neq i_2$.
\item Denote by $(~,~)_i$ the inner product  on $\Lambda_i\otimes_\Z\R$ induced by the Killing form of $\mathfrak q_i$. Then 
$c_i(~,~)_i=((~,~))$ on $\Lambda_i\otimes_\Z\R$ for some $c_i>0$.
\item Denote by $(~,~)$ the inner product  on $\Lambda\otimes_\Z\R$ induced by the Killing form of $\mathfrak g$. Then 
$\Lambda_{i_1}\otimes_\Z\R$ and $\Lambda_{i_2}\otimes_\Z\R$ of $\Lambda\otimes_\Z\R$ are orthogonal with respect to $(~,~)$ whenever $i_1\neq i_2$. Since the set of subspaces $\{\Lambda_i\otimes_\Z\R\}_i$ are pairwise orthogonal with respect to the positive definite inner products $((~,~))$ and $(~,~)$, we conclude that $((~,~))$ determines $(~,~)$ up to a positive factor on each $\Lambda_i\otimes_\Z\R $ for all $i$.
\end{enumerate}
\end{lemma}

\begin{proof}
Since $W$ preserves $\mathrm{Char}(\C^k)$, the weight subspaces $\Lambda_{i_1}\otimes_\Z\R$ and $\Lambda_{i_2}\otimes_\Z\R$ are orthogonal with respect to $((~,~))$. This proves (i).
Assertion (ii) follows from \cite[VI $\mathsection1$ Prop. 5 Cor. (i)]{Bou81}.
For (iii), by definition of the Killing form, the
weight subspaces $\Lambda_{i_1}\otimes_\Z\R$ and $\Lambda_{i_2}\otimes_\Z\R$ are orthogonal with respect to $(~,~)$.
The conclusion of (iii) then follows from (i) and (ii).
\end{proof}

The following result is obtained implicitly in \cite[$\mathsection2$]{Hui13}. Since it is crucial to Proposition \ref{rdss}, we make it explicit.

\begin{prop}\label{rootcomp}
If each simple factor $\mathfrak q_i$ of $\mathfrak g$ is of type $A_n$ for some $n\in\N\backslash\{1,2,3,5,7,8\}$ and $\mathfrak g$ has at most one $A_4$ factor (the conditions in Hypothesis \ref{A}), then $\mathfrak g$ is isomorphic to $\mathfrak g'$ and there is a one to one correspondence between the two multisets $\{\mathfrak q_i\}_i$ and $\{\mathfrak q_j'\}_j$, 
denoted by $\{\mathfrak q_i \leftrightarrow \mathfrak q_i'\}_i$ such that the following conditions hold:
\begin{enumerate}[($i$)]
\item $\Lambda_i\otimes_\Z\R=\Lambda_i'\otimes_\Z\R$ as subspace of $\Lambda\otimes_\Z\R$ for all $i$;
\item $\mathfrak g_i$ is isomorphic to $\mathfrak g_i'$ for all $i$;
\item $R_i=R_i'$ as a subset of $\Lambda\otimes_\Z\R$ for all $i$;
\item $R=R'$ as a subset of $\Lambda\otimes_\Z\R$.
\end{enumerate}
\end{prop}

\begin{proof}
Since $\mathfrak g\subset\End_k(\C)$ and $\mathfrak g'\subset\End_k(\C)$ have the same formal character 
$\mathfrak t\subset\End_k(\C)$ and the simple factors of $\mathfrak g$ satisfy the conditions in Hypothesis \ref{A},
$\mathfrak g$ and $\mathfrak g'$ are isomorphic \cite[Theorem 2.14, 2.17]{Hui13}.
%Let $V=\C^k$ and $\kappa:\End_k(\C)=\End(V)\to \End(V\otimes V^*)$ the natural map.
%To prove the proposition, it suffices to consider the faithful representations 
%$\kappa(\mathfrak{g}),\kappa(\mathfrak{g}')\subset\End(V\otimes V^*)$
%since the roots of $(\mathfrak{g},\mathfrak{t})$ (resp. $(\mathfrak{g}',\mathfrak{t}')$) 
%is the pull back of the roots of $(\kappa(\mathfrak{g}),\kappa(\mathfrak{t}))$ (resp. $(\kappa(\mathfrak{g}'),\kappa(\mathfrak{t}'))$)
%under $\kappa$. Hence, we are in the setting of \cite[$\mathsection2.2$]{Hui13} that all roots of $\mathfrak{g}$ and $\mathfrak{g}'$ appear in the formal character.
Let $u_j'\in R'_j$ be a root of $\mathfrak q'_j$ such that the orthogonal projection of $u_j'$ (with respect to $((~,~))$) to $\Lambda_i\otimes_\Z\R$ is nonzero.
Since $\mathfrak q_i=A_n$ with $n\geq 4$ and $\mathfrak g$ is of type A (hence the assumptions of \cite[$\mathsection2.10$]{Hui13} are fulfilled), 
we have
$$u_j'\notin (\Lambda_i\otimes_\Z\R) \cup (\Lambda_i\otimes_\Z\R)^\bot$$
 only if $\mathfrak g$ has a $A_n$ factor where $n\in\{1,2,5,7\}$ or $\mathfrak g$ has two $A_4$ factors \cite[Proposition 2.11]{Hui13}.
Since these cases are excluded, we obtain
$$u_j'\in \Lambda_i\otimes_\Z\R.$$
Since the root system 
$$(\Lambda_j'\otimes_\Z\R,~R_j', ~(~,~)')$$
 of $\mathfrak q_j'$ \cite[$\mathsection21.1$]{FH91} is irreducible,
we obtain $R_j'\subset \Lambda_i\otimes_\Z\R$ by Lemma \ref{Killing}(iii).
Thus, we have $\Lambda_j'\otimes_\Z\R \subset \Lambda_i\otimes_\Z\R$.
Since the number of simple factors of $\mathfrak g$ and $\mathfrak g'$ are equal (because $\mathfrak g\cong\mathfrak g'$)
and $R$ (resp. $R'$)  generates vector space $\Lambda\otimes_\Z\R$,
we conclude (i)
$$\Lambda_j'\otimes_\Z\R = \Lambda_i\otimes_\Z\R$$
and thus obtain an one to one correspondence $\{\mathfrak q_i \leftrightarrow \mathfrak q_i'\}_i$ such that (ii) holds (because $\dim\mathfrak q_i=\dim\mathfrak q_i'$).
Since $\mathfrak g$ and $\mathfrak g'$ are isomorphic and satisfy the simple factor conditions in Hypothesis \ref{A}, we obtain 
$R_i\subset R_i'$ and $R_i'\subset R_i$ by 
$$\Lambda_i'\otimes_\Z\R = \Lambda_i\otimes_\Z\R$$
 and \cite[$\mathsection2.13$]{Hui13}.
We conclude that $R_i=R_i'$ for all $i$, which is (iii). Then (iv) follows from (iii).
\end{proof}

\subsection{The root datum and conjugacy class of $\bG_\ell$}\label{s3.2}

Let $F$ be a field with $\bar F$ an algebraic closure. To each pair $(\bG^{\mathrm{sp}},\bT^{\mathrm{sp}})$ where $\bG^{\mathrm{sp}}$ is a connected split reductive group defined over $F$ and $\bT^{\mathrm{sp}}$ is a split maximal torus of $\bG^{\mathrm{sp}}$, one associates a root datum $\Psi=\psi(\bG^{\mathrm{sp}},\bT^{\mathrm{sp}})=(\X,R,\X^\vee,R^\vee)$ as follows (\cite[Chapter 15]{Sp08}, \cite[$\mathsection2$ ($F=\bar F$)]{Sp79}). 
Denote by $\X$ the character group of $\bT^{\mathrm{sp}}$ and by $\X^\vee$ the cocharacter group of $\bT^{\mathrm{sp}}$. They are free abelian groups of 
 rank equal to the dimension of $\bT^{\mathrm{sp}}$ and admit a natural pairing $\left\langle~,~\right\rangle$: if $x\in\X$ and $u\in\X^\vee$, then $x(u(t))=t^{\left\langle x,u \right\rangle}$ for $t\in \bar F^*$.
Take $R$ to be the roots of $\bG^{\mathrm{sp}}$ (the non-zero characters of the adjoint representation of $\bG^{\mathrm{sp}}$)
with respect to $\bT^{\mathrm{sp}}$. For $\alpha\in R$, let $\bT_\alpha^{\mathrm{sp}}$ be the identity component of the kernel of $\alpha$ and $\bG_\alpha^{\mathrm{sp}}$ the derived group of the centralizer of $\bT_\alpha^{\mathrm{sp}}$ in $\bG^{\mathrm{sp}}$. Then $\bG_\alpha^{\mathrm{sp}}$ is semisimple of rank $1$ and there is a unique homomorphism $\alpha^\vee: F^*\to\bG_\alpha^{\mathrm{sp}}$ such that $\bT^{\mathrm{sp}}=(\mathrm{Im}\alpha^\vee)\bT_\alpha^{\mathrm{sp}}$ and $\left\langle\alpha,\alpha^\vee\right\rangle=2$. These $\alpha^\vee$ make up $R^\vee$. A \emph{central isogeny} \cite[$\mathsection 9.6.3$]{Sp08} $\phi$ of $(\bG^{\mathrm{sp}},\bT^{\mathrm{sp}})$ onto $((\bG^{\mathrm{sp}})',(\bT^{\mathrm{sp}})')$ induces an \emph{isogeny} of root data \cite[$\mathsection1$]{Sp79},
$$f(\phi):\psi((\bG^{\mathrm{sp}})',(\bT^{\mathrm{sp}})')\to\psi(\bG^{\mathrm{sp}},\bT^{\mathrm{sp}}).$$

\begin{thm}\cite[Theorem 16.3.3, 16.3.2]{Sp08}, \cite[Theorem 2.9 ($F=\bar F$)]{Sp79}\label{rdisom}
\begin{enumerate}[($i$)]
\item For any root datum $\Psi$ with reduced root system, there exists a connected split reductive group $\bG^{\mathrm{sp}}$ and a maximal split torus $\bT^{\mathrm{sp}}$ in $\bG^{\mathrm{sp}}$ such that $\Psi=\psi(\bG^{\mathrm{sp}},\bT^{\mathrm{sp}})$. The pair $(\bG^{\mathrm{sp}},\bT^{\mathrm{sp}})$ is unique up to isomorphism.
\item Let $\Psi=\psi(\bG^{\mathrm{sp}},\bT^{\mathrm{sp}})$ and $\Psi'=\psi((\bG^{\mathrm{sp}})',(\bT^{\mathrm{sp}})')$. If $f$ is an isogeny of $\Psi'$ into $\Psi$, then there exists a central isogeny $\phi$ of $(\bG^{\mathrm{sp}},\bT^{\mathrm{sp}})$ onto $((\bG^{\mathrm{sp}})',(\bT^{\mathrm{sp}})')$ with $f(\phi)=f$. Two such $\phi$ differ by an inner automorphism $\mathrm{Int}(t)$ of $\bG^{\mathrm{sp}}$, where 
$t\in\bT^{\mathrm{sp}}(F)$.
\end{enumerate}
\end{thm}

\begin{remark}
If $F=\bar F$, then every connected reductive $\bG$ over $F$ splits. 
\end{remark}

Let $\ell$ and $\ell'$ be two distinct prime numbers. We identify $\bG_\ell\times_{\Q_\ell}\C$ and $\bG_{\ell'}\times_{\Q_{\ell'}}\C$ as 
connected reductive subgroups of $\GL_{k,\C}$.
By Theorem \ref{bicharC},
 the chain 
\begin{equation}\label{bi}
\bT_\C^{\ss}\subset\bT_\C\subset\GL_{k,\C}
\end{equation}
is isomorphic to the formal bi-characters of both $\bG_\ell\times_{\Q_\ell}\C$ and $\bG_{\ell'}\times_{\Q_{\ell'}}\C$. 
Hence, up to conjugation in $\GL_k(\C)$, we may assume
$\bT_\C$ is a common maximal torus of $\bG_\ell\times_{\Q_\ell}\C$ and $\bG_{\ell'}\times_{\Q_{\ell'}}\C$ and
$\bT_\C^{\ss}$ is a common maximal torus of $\bG_\ell^{\der}\times_{\Q_\ell}\C$ and $\bG_{\ell'}^{\der}\times_{\Q_{\ell'}}\C$ 
(the derived groups of $\bG_\ell\times_{\Q_\ell}\C$ and $\bG_{\ell'}\times_{\Q_{\ell'}}\C$) .

\begin{definition}
Define the following notations.
\begin{enumerate}[($a$)]
\item $\X$: the character group of $\bT_\C$
\item $\X^\vee$: the cocharacter group of $\bT_\C$
\item $R$: the roots of $\bG_\ell\times_{\Q_\ell}\C$ with respect to $\bT_\C$
\item $R^\vee$: the coroots of  $\bG_\ell\times_{\Q_\ell}\C$ with respect to $\bT_\C$
\item $R'$: the roots of $\bG_{\ell'}\times_{\Q_{\ell'}}\C$ with respect to $\bT_\C$
\item $(R')^\vee$: the coroots of  $\bG_{\ell'}\times_{\Q_{\ell'}}\C$ with respect to $\bT_\C$
\item $\X^{\ss}$: the character group of $\bT_\C^{\ss}$
\item $(\X^{\ss})^\vee$: the cocharacter group of $\bT_\C^{\ss}$
\item $R|_{\bT^{\ss}_\C}$: the roots of $\bG_\ell^{\der}\times_{\Q_\ell}\C$ with respect to $\bT_\C^{\ss}$
\item $R^\vee$: the coroots of  $\bG_\ell^{\der}\times_{\Q_\ell}\C$ with respect to $\bT_\C^{\ss}$
\item $R'|_{\bT_\C^{\ss}}$: the roots of $\bG_{\ell'}^{\der}\times_{\Q_{\ell'}}\C$ with respect to $\bT_\C^{\ss}$
\item $(R')^\vee$: the coroots of  $\bG_{\ell'}^{\der}\times_{\Q_{\ell'}}\C$ with respect to $\bT_\C^{\ss}$
\end{enumerate}
\end{definition}

\begin{remark}\label{same}
The definitions of (i),(j),(k),(l) make sense. Indeed, 
there are natural maps $\X\to\X^{\ss}$ and 
$(\X^{\ss})^\vee\subset \X^\vee$ because $\bT_\C^{\ss}$ is a subtorus of $\bT_\C$. 
Notations (i) and (k) come from the restriction of $R$ to $\bT_\C^{\ss}$. 
Notations (j) and (l) come from the fact that the 
coroots of $(\bG_\ell\times_{\Q_\ell}\C,\bT_\C)$ and $(\bG_\ell^{\der}\times_{\Q_\ell}\C,\bT_\C^{\ss})$
are identical.
\end{remark}

Let $\mathfrak{g}_\ell^{\der}$, $\mathfrak{g}_{\ell'}^{\der}$, and $\mathfrak{t}$ be the Lie algebras of 
$\bG_\ell^{\der}\times_{\Q_\ell}\C$, $\bG_{\ell'}^{\der}\times_{\Q_{\ell'}}\C$, and $\bT_\C^{\ss}$  respectively. 
Then $\mathfrak{t}$ is a common Cartan subalgebra of $\mathfrak{g}_\ell^{\der}$ and $\mathfrak{g}_{\ell'}^{\der}$.

\begin{prop}\label{rdss}
If  $R|_{\bT^{\ss}_\C}=R'|_{\bT^{\ss}_\C}$, then $R^\vee=(R')^\vee$. Therefore,
the root data $(\X^{\ss},R|_{\bT^{\ss}_\C},(\X^{\ss})^\vee,R^\vee)$ and $(\X^{\ss},R'|_{\bT^{\ss}_\C},(\X^{\ss})^\vee,(R')^\vee)$ of respectively
$(\bG_\ell^{\der}\times_{\Q_\ell}\C,\bT_\C^{\ss})$ and $(\bG_{\ell'}^{\der}\times_{\Q_{\ell'}}\C,\bT_\C^{\ss})$ are equal.
\end{prop}

\begin{proof}
%The conditions imply that $R|_{\bT^{\ss}_\C}=R'|_{\bT^{\ss}_\C}$.
For any complex Lie group homomorphism $\phi$, denote by $d\phi$ the differential of $\phi$ at identity.
Let $\alpha\in R|_{\bT^{\ss}_\C}=R'|_{\bT^{\ss}_\C}$,
$\alpha^\vee\in R^\vee$ and $(\alpha')^\vee\in (R')^\vee$ be the coroots corresponding to $\alpha$.
Then 
$$(d\alpha:\mathfrak{t}\to \C)\in\mathfrak{t}^*$$
 is a root of $\mathfrak{g}_\ell^{\der}$ as well as a root of $\mathfrak{g}_{\ell'}^{\der}$.
If we identify
$$d\alpha^\vee:\C\to \mathfrak{t}\hspace{.2in} \mathrm{and}\hspace{.2in}d(\alpha')^\vee:\C\to \mathfrak{t}$$
 as elements of $\mathfrak{t}$ by the images of $1\in\C$,
then by construction they are 
\emph{distinguished elements} of $\mathfrak{t}$ \cite[$\mathsection14.1$]{FH91}
corresponding to the root $d\alpha$ of $\mathfrak{g}_\ell^{\der}$ and $\mathfrak{g}_{\ell'}^{\der}$ respectively.
Let $(~,~)$ and $(~,~)'$ on $\mathfrak{t}^*$ be the inner products induced by
 the Killing forms of $\mathfrak{g}_\ell^{\der}$ and $\mathfrak{g}_{\ell'}^{\der}$  respectively.
For $\beta\in  R|_{\bT^{\ss}_\C}=R'|_{\bT^{\ss}_\C}$,
we obtain by \cite[Corollary 14.29]{FH91} that
\begin{align}\label{angle}
\begin{split}
 \left\langle \beta,\alpha^\vee\right\rangle
&=d\beta(d\alpha^\vee)
=\frac{2(d\beta,d\alpha)}{(d\alpha,d\alpha)}
=\frac{2||d\beta||\cos\theta}{||d\alpha||},\\
\left\langle \beta,(\alpha')^\vee\right\rangle
&=d\beta(d(\alpha')^\vee)
=\frac{2(d\beta,d\alpha)'}{(d\alpha,d\alpha)'}
=\frac{2||d\beta||'\cos\theta'}{||d\alpha||'},
\end{split}
\end{align}
where $\theta$ and  $||\cdot||$ (resp. $\theta'$ and $||\cdot||'$) denote the angle between $d\alpha$ and $d\beta$ 
and the length under the inner product $(~,~)$  (resp. the inner product $(~,~)'$).
Let $V_\R$ be the $\R$-span in $\mathfrak{t}^*$ of the common set of roots 
$$\{d\beta:~\beta\in R|_{\bT^{\ss}_\C}=R'|_{\bT^{\ss}_\C}\}$$
of $(\mathfrak{g}_\ell^{\der}, \mathfrak{t})$ and $(\mathfrak{g}_{\ell'}^{\der}, \mathfrak{t})$. 
Then $(~,~)$ and $(~,~)'$ are positive definite 
on $V_\R$ and define two root systems.
In particular, the two Weyl group (of $\mathfrak g_\ell^{\der}$ and $\mathfrak g_{\ell'}^{\der}$) actions on $V_\R$ are orthogonal for both $(~,~)$ and $(~,~)'$. Thus,
$(~,~)|_{V_\R}$ determines $(~,~)'|_{V_\R}$
up to a positive scalar factor on each irreducible root subsystem by Lemma \ref{Killing}(iii).
Hence, $\theta=\theta'$ always holds and 
$$\frac{||d\beta||}{||d\alpha||}=\frac{||d\beta||'}{||d\alpha||'}$$
 if $d\alpha$ and $d\beta$ belong to the same irreducible subsystem.
We conclude that  $\left\langle \beta,\alpha^\vee\right\rangle=\left\langle \beta,(\alpha')^\vee\right\rangle$
by (\ref{angle})
for all $\beta\in R|_{\bT^{\ss}_\C}=R'|_{\bT^{\ss}_\C}$.
Since $R|_{\bT^{\ss}_\C}=R'|_{\bT^{\ss}_\C}$ spans $\X^{\ss}\otimes_\Z\R$, we have $\alpha^\vee=(\alpha')^\vee$ in $(\X^{\ss})^\vee$.
Hence, $R^\vee=(R')^\vee$.
\end{proof}

\begin{thm}\label{generalrd}
If  $R|_{\bT^{\ss}_\C}=R'|_{\bT^{\ss}_\C}$, then $R=R'$. Therefore,
the root data $\Psi=(\X,R,(\X)^\vee,R^\vee)$ and $\Psi'=(\X,R',(\X)^\vee,(R')^\vee)$ of respectively
$(\bG_\ell\times_{\Q_\ell}\C,\bT_\C)$ and $(\bG_{\ell'}\times_{\Q_{\ell'}}\C,\bT_\C)$ are equal.
\end{thm}

\begin{proof}
By Remark \ref{same} and Proposition \ref{rdss}, the coroots of $\Psi$ and $\Psi'$ are the same, i.e., $R^\vee=(R')^\vee$. 
It suffices to 
prove $R=R'$. Let $\X_\R=\X\otimes_\Z\R$. The formal character $\bT_\C\subset\GL_{k,\C}$ corresponds to a finite 
subset $S$ of $\X$ which spans $\X_\R$. The subgroup $G_S$ of $\GL(\X_\R)$ that preserves $S$ is finite and contains 
the Weyl groups $W$ and $W'$ of $(\bG_\ell\times_{\Q_\ell}\C,\bT_\C)$ and $(\bG_{\ell'}\times_{\Q_{\ell'}}\C,\bT_\C)$ respectively. 
By Weyl's unitarian trick,
 there exists a positive definite
inner product $(~,~)$ on $\X_\R$ such that $G_S$ is orthogonal. 
Denote by $V_\R$ and $V_\R'$ the $\R$-spans of 
$R$ and $R'$ in $\X_\R$. 
Denote by $U_\R$ the $\R$-span of the characters of $\X$ that annihilate $\bT_\C^{\ss}$. 
We obtain
\begin{equation}\label{complement}
V_\R\oplus U_\R=\X_\R=V_\R'\oplus U_\R.
\end{equation}
Let $V_\R^\bot$ (resp. $(V_\R')^\bot$) be the orthogonal complement of $V_\R$ (resp. $V_\R'$) in $\X_\R$. 
Since $V_\R$ and $V_\R^\bot$ (resp. $V_\R'$ and $(V_\R')^\bot$) are both invariant under $W$ (resp. $W'$), 
the action of $W$ (resp. $W'$) on $V_\R$ (resp. $V_\R'$) does not contain any trivial subrepresentation, 
and $W$ (resp. $W'$) is identity on $U_\R$,
we obtain $V_\R^\bot=U_\R=(V_\R')^\bot$ by (\ref{complement}), and hence, the following 
\begin{equation}\label{orth}
V_\R = U_\R^\bot = V_\R'.
\end{equation}
For any $\gamma\in R^\vee=(R')^\vee$, let $v_\gamma$ be the unique element in $\X_\R$ such that 
$$(\alpha,v_\gamma)=\left\langle\alpha,\gamma\right\rangle$$
for all $\alpha\in\X$.
Since the images of the coroots generate $\bT_\C^{\ss}$, we obtain
\begin{equation}\label{span}
\mathrm{Span}_\R\{v_\gamma:\gamma\in R^\vee=(R')^\vee\}=U_\R^\bot.
\end{equation}
The natural map $R\to R|_{\bT_\C^{\ss}}$ ($R'\to R'|_{\bT_\C^{\ss}}$) is a bijection and 
$R|_{\bT_\C^{\ss}}=R'|_{\bT_\C^{\ss}}$ is the hypothesis. Let $\alpha\in R$ and $\alpha'\in R'$ be two roots 
such that $\alpha|_{\bT_\C^{\ss}}=\alpha'|_{\bT_\C^{\ss}}$. Then
$$(\alpha,v_\gamma)=\left\langle\alpha,\gamma\right\rangle=\left\langle\alpha|_{\bT_\C^{\ss}},\gamma\right\rangle=\left\langle\alpha'|_{\bT_\C^{\ss}},\gamma\right\rangle    =\left\langle\alpha',\gamma\right\rangle=(\alpha',v_\gamma)$$ 
for all $\gamma\in R^\vee=(R')^\vee$.
Therefore, we obtain $\alpha=\alpha'$ by (\ref{orth}) and (\ref{span}), which implies $R=R'$.
\end{proof}

\begin{cor}\label{generalconj}
If  $R|_{\bT^{\ss}_\C}=R'|_{\bT^{\ss}_\C}$, then the complex reductive subgroups $\bG_\ell\times_{\Q_\ell}\C$ and $\bG_{\ell'}\times_{\Q_{\ell'}}\C$ of $\GL_{k,\C}$ 
are conjugate in $\GL_{k,\C}$.
\end{cor}

\begin{proof}
Since the root data $\Psi$ and $\Psi'$ are equal, this defines an isomorphism $f:\Psi'\to\Psi$ of root data.
By Theorem \ref{rdisom}(ii),
there exists an isomorphism $\phi:(\bG_\ell\times_{\Q_\ell}\C,\bT_\C)\to (\bG_{\ell'}\times_{\Q_{\ell'}}\C,\bT_\C)$ such that $f(\phi)=f$.
This implies the that standard representation $\bG_\ell\times_{\Q_\ell}\C\subset\GL_{k,\C}$ and the representation
$\bG_\ell\times_{\Q_\ell}\C\stackrel{\phi}{\rightarrow}\bG_{\ell'}\times_{\Q_{\ell'}}\C\subset\GL_{k,\C}$ of 
$\bG_\ell\times_{\Q_\ell}\C$
have the same character. Hence, the two representations are equivalent and the images are conjugate in $\GL_{k,\C}$.
\end{proof}

Since Proposition \ref{rootcomp}(iv) implies $R|_{\bT^{\ss}_\C}=R'|_{\bT^{\ss}_\C}$, we obtain the following 
immediately by Theorem \ref{generalrd} and Corollary \ref{generalconj}.

\begin{thm}\label{rdr}
If Hypothesis \ref{A} is satisfied, then the root datum of $(\bG_\ell\times_{\Q_\ell}\C, \bT_\C)$ and the conjugacy class of $\bG_\ell\times_{\Q_\ell}\C$ in $\GL_{k,\C}$ are independent of $\ell$.
\end{thm}

\begin{remark}
The formal bi-characters of $\bG\subset\GL_{k,\C}$ does not determine $\bG$ even if it is of type A:
Let $\bG=\SL_{2,\C}$ (semisimple) and $V$ the standard representation of $\bG$. Denote by $\mathrm{Sym}^iV$ the $i$th symmetric 
power of $V$ ($\mathrm{Sym}^0V$ denotes the trivial representation).
Let $\G_{m,\C}^3\subset\GL_{3,\C}$ the diagonal subgroup and consider the following $3$-dimensional representations of $\bG$:
\begin{align*}
\rho_1:&=\mathrm{Sym}^0(V)\oplus \mathrm{Sym}^1(V).\\
\rho_2:&=\mathrm{Sym}^2(V). 
\end{align*}
The images $\rho_1(\bG)\cong\SL_{2,\C}$ and $\rho_2(\bG)\cong\PSL_{2,\C}$ viewed as
subgroups of $\GL_{3,\C}$ have the same formal character (bi-character)
$$\{(1,z,z^{-1})\in\G_{m,\C}^3\subset\GL_{3,\C}: z\in\C^*\}$$
but they are not similar in $\GL_{3,\C}$ (not even isomorphic).
\end{remark}

\section{Forms of reductive groups}\label{s4}

Let $\bG^{\mathrm{sp}}$ be a connected split reductive group over the field $F$. 
Let $\bT^{\mathrm{sp}}$ be a maximal split $F$-subtorus of $\bG^{\mathrm{sp}}$, $W$ the Weyl group with respect to $\bT^{\mathrm{sp}}$, $\bN$ the normalizer of $\bT^{\mathrm{sp}}$ in $\bG^{\mathrm{sp}}$, and $\bB$ an $F$-Borel subgroup containing $\bT^{\mathrm{sp}}$. Let $\bC$ be the center of $\bG^{\mathrm{sp}}$.
The automorphism group $\Aut_{\bar F}\bG^{\mathrm{sp}}$
of $\bG^{\mathrm{sp}}\times_F\bar F$ is acted on by $\Gal_F$ in the following way.\\
If $\alpha\in\Aut_{\bar F}\bG^{\mathrm{sp}}$ and $\sigma\in\Gal_F$, then ${}^\sigma\! \alpha\in \Aut_{\bar F}\bG^{\mathrm{sp}}$ so that
\begin{equation}\label{gal}
{}^\sigma\! \alpha(x):=\sigma(\alpha(\sigma^{-1} x))\hspace{.1in}\forall x\in\bG^{\mathrm{sp}}(\bar F).
\end{equation}
The group $\Aut_{\bar F}\bG^{\mathrm{sp}}$ admits a short exact sequence of $\Gal_F$-groups \cite[Corollary 2.14]{Sp79} (see also \cite[XXIV Theorem 1.3]{DG62})
\begin{equation}\label{ses1}
1\to \mathrm{Inn}_{\bar F}\bG^{\mathrm{sp}}\to \Aut_{\bar F}\bG^{\mathrm{sp}}\to \mathrm{Out}_{\bar F}\bG^{\mathrm{sp}}\to 1,
\end{equation}
where $\mathrm{Inn}_{\bar F}\bG^{\mathrm{sp}}$, the inner automorphism group is naturally 
isomorphic to the group of $\bar F$-points of $\bG^{\ad}:=\bG^{\mathrm{sp}}/\bC$ the adjoint quotient of $\bG^{\mathrm{sp}}$ 
and $\mathrm{Out}_{\bar F}\bG^{\mathrm{sp}}$, the outer automorphism group
is acted on trivially by $\Gal_F$ because $\bG^{\mathrm{sp}}$ is split.

\begin{prop} \label{4.1}
The group $\Aut_{\bar F}\bG^{\mathrm{sp}}$ contains a 
$\Gal_F$-invariant subgroup that preserves $\bT^{\mathrm{sp}}$ and $\bB$ and is mapped isomorphically onto $\mathrm{Out}_{\bar F}\bG^{\mathrm{sp}}$.
Hence, (\ref{ses1}) is a split short exact sequence of $\Gal_F$-groups.
\end{prop}

\begin{proof} Let $\Delta$ be the set of simple roots with respect to $(\bT^{\mathrm{sp}},\bB)$. 
Let $\mathbf{U}_\alpha$ be the root subgroup for $\alpha\in\Delta$ (the construction of Chevalley). 
It is isomorphic 
to the $F$-affine line. 
Choose $u_\alpha\in \mathbf{U}_\alpha(F)\backslash\{0\}$ for all $\alpha\in \Delta$.
Then the subgroup of $\Aut_{\bar F}\bG^{\mathrm{sp}}$ that leave $\bT^{\mathrm{sp}}$, $\bB$, and $\{u_\alpha\}_{\alpha\in\Delta}$ 
invariant is mapped isomorphically onto $\mathrm{Out}_{\bar F}\bG^{\mathrm{sp}}$ by \cite[Proposition 2.13, Corollary 2.14]{Sp79}.
This subgroup is
$\Gal_F$-invariant 
since $\bT^{\mathrm{sp}}$, $\bB$, and $\{u_\alpha\}_{\alpha\in\Delta}$ are $\Gal_F$-invariant.
\end{proof}

We then obtain a split short exact sequence of pointed sets by non-abelian cohomology \cite[$\mathsection5$]{Se97} 
\begin{equation}\label{ses2}
1\to (H^1(F,\mathrm{Inn}_{\bar F}\bG^{\mathrm{sp}}),0')\stackrel{i}{\rightarrow}  (H^1(F,\Aut_{\bar F}\bG^{\mathrm{sp}}),0)\stackrel{\pi}{\rightarrow} 
(H^1(F,\mathrm{Out}_{\bar F}\bG^{\mathrm{sp}}),0'')\to 1,
\end{equation}
where $0',0,0''$ denote the \emph{neutral elements} \cite[$\mathsection5.1$]{Se97}.
This means that $i(0')=0$, $\pi(0)=0''$, $i$ is injective, $\pi$ is surjective, $\pi^{-1}(0'')=\mathrm{Im}(i)$ \cite[$\mathsection5.4,5.5$]{Se97}, and there is a pointed map 
$j: (H^1(F,\mathrm{Out}_{\bar F}\bG^{\mathrm{sp}}),0'') \to (H^1(F,\Aut_{\bar F}\bG^{\mathrm{sp}}),0)$ such that $\pi\circ j=\mathrm{Id}$. 

The elements of $H^1(F,\Aut_{\bar F}\bG^{\mathrm{sp}})$ are in bijective correspondence with the $F$-forms of $\bG^{\mathrm{sp}}$ \cite[Chapter 3.1]{Se97}.
If $\bG$ is an $F$-form of $\bG^{\mathrm{sp}}$, then there exists an $\bar F$-isomorphism $\phi:\bG^{\mathrm{sp}}\times_F\bar F\to \bG\times_F\bar F$. The isomorphism class of $\bG/F$ is represented by $[c_\sigma]\in H^1(F,\Aut_{\bar F}\bG^{\mathrm{sp}})$, where 
\begin{equation}\label{action1}
c_\sigma(x):=\phi^{-1}(\sigma\phi(\sigma^{-1}(x))) \hspace{.1in}\forall x\in\bG^{\mathrm{sp}}(\bar F).
\end{equation}
 Two forms 
$\bG'$ and $\bG''$ that map to the same image in $H^1(F,\mathrm{Out}_{\bar F}\bG^{\mathrm{sp}})$ are \emph{inner twists} of each other \cite[Chapter I $\mathsection5.5$ Corollary 2]{Se97}, i.e., $[\bG'']\in H^1(F,\mathrm{Inn}_{\bar F}\bG')$. The following result is well-known (see for example \cite[Chapter X $\mathsection2$]{CF65}, \cite[XXIV Theorem 3.11]{DG62}). We supply a proof that we learnt from \cite[Proposition 29.4]{Gi08}.

\begin{thm}\label{4.2}
The set $H^1(F,\mathrm{Out}_{\bar F}\bG^{\mathrm{sp}})$ in (\ref{ses2}) is in one-to-one correspondence with the set of quasi-split $F$-forms of $\bG^{\mathrm{sp}}$.
\end{thm}

\begin{proof} Let $[c_\sigma]$ be an element of $H^1(F,\mathrm{Out}_{\bar F}\bG^{\mathrm{sp}})$. Then we obtain by Proposition \ref{4.1}
an element $[c'_\sigma:=j(c_\sigma)]\in \pi^{-1}([c_\sigma])$ such that $c'_\sigma\in\Aut_{\bar F}\bG^{\mathrm{sp}}$ preserves $\bT^{\mathrm{sp}}$ and $\bB$ and 
is invariant under $\Gal_F$ for all $\sigma\in\Gal_F$. The $F$-form $\bG'$ corresponding to $[c'_\sigma]$ is obtained by 
defining an $F$-structure on $\bG^{\mathrm{sp}}(\bar F)$ by the twisted Galois action:
\begin{equation*}
\sigma\cdot x:=c'_\sigma(\sigma x)\hspace{.1in} \forall \sigma\in\Gal_F,~x\in \bG^{\mathrm{sp}}(\bar F).
\end{equation*}  
Since $\bB(\bar F)$ is invariant under $\sigma$ and $c'_\sigma$ for all $\sigma\in\Gal_F$, $\bG'$ has a Borel subgroup defined over $F$.
Hence, the quasi-split $F$-forms of $\bG^{\mathrm{sp}}$ surject onto $H^1(F,\mathrm{Out}_{\bar F}\bG^{\mathrm{sp}})$.

Let $\bG'$ and $\bG''$ be two quasi-split $F$-forms of $\bG^{\mathrm{sp}}$ that map to the same image via $\pi$. 
They differ by an inner twist $[c_\sigma]\in H^1(F,\mathrm{Inn}_{\bar F}\bG')$.
Let $\bT'\subset\bB'$ (resp. $\bT''\subset\bB''$) be an embedding of an $F$-maximal torus of $\bG'$ (resp. $\bG''$)
 in an $F$-Borel subgroup of $\bG'$ (resp. $\bG''$).
Let $\mathbf{C}'$ be the center of $\bG'$ and $\Delta'$ the simple roots of $\bG'$ with respect to $(\bT',\bB')$. We may assume $c_\sigma\in \bT'/\mathbf{C}'$ for all $\sigma\in\Gal_F$ \cite[Proposition 2.5(ii)]{Sp79}.
Since $\Gal_F$ permutes $\Delta'$ which is a basis of characters of $\bT'/\mathbf{C}'$, torus $\bT'/\mathbf{C}'$
is a direct sum of induced tori, i.e., there exist finite separable extensions $F_1,...,F_k$ of $F$ such that
\begin{equation*}
\bT'/\mathbf{C}'=\bigoplus^k_{i=1}\mathrm{Ind}^F_{F_i}\G_{m,F_i}.
\end{equation*}
By Shapiro's lemma and Hilbert's Theorem $90$, we obtain $H^1(F, \bT'/\mathbf{C}')=0$. Therefore, $\bG'$ and $\bG''$ are $F$-isomorphic
and we conclude that the quasi-split $F$-forms of $\bG^{\mathrm{sp}}$ map bijectively onto $H^1(F,\mathrm{Out}_{\bar F}\bG^{\mathrm{sp}})$.
\end{proof}

Let
$\Aut_{\bar F,\bT^{\mathrm{sp}}}\bG^{\mathrm{sp}}$ be the subgroup of $\Aut_{\bar F}\bG^{\mathrm{sp}}$ 
that preserves $\bT^{\mathrm{sp}}$. Denote by $\Aut_{\bar F}\bT^{\mathrm{sp}}$ the automorphism group of $\bT^{\mathrm{sp}}\times_F\bar F$.
Although the following proposition is contained in \cite[XXIV Proposition 2.6]{DG62}, we still provide a proof.

\begin{prop}\label{4.3}
With the notations introduced above, the following commutative diagram of $\Gal_F$-groups has exact rows and columns.
The maps from the top row to the middle row are given by 
inner automorphisms by elements of $\bT^{\mathrm{sp}}(\bar F)$ and the first two maps from
the middle row to the bottom row are given by the restriction to $\bT^{\mathrm{sp}}$, i.e., $\Omega_{\bar F}:=\Aut_{\bar F,\bT^{\mathrm{sp}}}\bG^{\mathrm{sp}}/\bT^{\mathrm{sp}}(\bar F)$ can be identified as a subgroup of $\Aut_{\bar F}\bT^{\mathrm{sp}}$.
\begin{equation}\label{ses3}
\begin{aligned}
\xymatrix{
1\ar[r] &\bT^{\mathrm{sp}}(\bar F) \ar[d] \ar[r] &\bT^{\mathrm{sp}}(\bar F) \ar[d]\ar[r] &1\ar[d]\ar[r] & 1\\
1\ar[r] &\bN/\bC(\bar F) \ar[d] \ar[r] &\Aut_{\bar F,\bT^{\mathrm{sp}}}\bG^{\mathrm{sp}} \ar[d]^{\mathrm{Res}} \ar[r]& \mathrm{Out}_{\bar F}\bG^{\mathrm{sp}}\ar[d]^{\mathrm{Id}}\ar[r]& 1\\
1\ar[r] & W\ar[r] & \Omega_{\bar F}:=\Aut_{\bar F,\bT^{\mathrm{sp}}}\bG^{\mathrm{sp}}/\bT^{\mathrm{sp}}(\bar F)\ar[r] & \mathrm{Out}_{\bar F}\bG^{\mathrm{sp}} \ar[r] &1}
\end{aligned}
\end{equation}
\end{prop}

\begin{proof}
It is clear that the diagram is commutative and the rows and columns are exact. 
The only thing one needs to show is that $\Aut_{\bar F,\bT^{\mathrm{sp}}}\bG^{\mathrm{sp}}/\bT^{\mathrm{sp}}(\bar F)$
embeds into $\Aut_{\bar F}\bT^{\mathrm{sp}}$ by restricting automorphisms in $\Aut_{\bar F,\bT^{\mathrm{sp}}}\bG^{\mathrm{sp}}$ to 
the maximal
torus $\bT^{\mathrm{sp}}$.
For any $\alpha\in \Aut_{\bar F,\bT^{\mathrm{sp}}}\bG^{\mathrm{sp}}$, write $\alpha=\beta\gamma$ where $\beta\in\bN/\bC(\bar F)$ and $\gamma$
fixes $\bT^{\mathrm{sp}}$ and $\bB$ by the splitting of Proposition \ref{4.1}.
If $\alpha$ is trivial in $\Aut_{\bar F}\bT^{\mathrm{sp}}$, then $\beta=\gamma^{-1}$ on $\bT^{\mathrm{sp}}$.
Since $W$ acts simply transitively on the Weyl chambers and $\gamma$ fixes the chamber corresponding
to $\bB$, $\beta$ belongs to the image of $\bT^{\mathrm{sp}}(\bar F)$. 
This implies $\gamma$ is trivial on $\bT^{\mathrm{sp}}$ and thus $\alpha=\beta$. 
\end{proof}

\begin{remark}\label{GTform}
The elements of $H^1(F,\Aut_{\bar F,\bT^{\mathrm{sp}}}\bG^{\mathrm{sp}})$ are in bijective correspondence with the
 $F$-forms of $(\bG^{\mathrm{sp}},\bT^{\mathrm{sp}}) $, i.e., the $F$-reductive groups $\bG$ together with an $F$-maximal torus $\bT$
such that after extending scalars to $\bar F$, there exists an $\bar F$-isomorphism 
\begin{equation*}
\phi:\bG^{\mathrm{sp}}\times_F\bar F\to \bG\times_F\bar F
\end{equation*} 
taking $\bT^{\mathrm{sp}}\times_F\bar F$ onto $\bT\times_F\bar F$.
The isomorphism class of $(\bG,\bT)$ is then represented by $[c_\sigma]\in H^1(F,\Aut_{\bar F,\bT^{\mathrm{sp}}}\bG^{\mathrm{sp}})$, where 
\begin{equation*}
c_\sigma(x):=\phi^{-1}(\sigma\phi(\sigma^{-1}(x))) \hspace{.1in}\forall x\in\bG^{\mathrm{sp}}(\bar F).
\end{equation*}
\end{remark}

\section{Proofs of the main results}\label{s5}
\subsection{Proof of Theorem \ref{main}}
We obtain Theorem \ref{main}(i) by Theorem \ref{rdr}.
The proof of Theorem \ref{main}(ii) consists of several ingredients which will be established separately. 
Lemmas \ref{lem2} and \ref{lem3} below are special cases of \cite[Proposition 10]{Vi96} and \cite[Theorem 1.1]{Ma09}.

\begin{lemma}\label{lem2}
Let $\bG$ be a connected reductive group over $\bar\Q$. Then there is a bijective correspondence from the equivalence classes
of finite dimensional $\bar\Q$-representations of $\bG$ to the equivalence classes of finite dimensional $\C$-representations 
of $\bG\times_{\bar\Q}\C$ given by base change $i:\bar\Q\to\C$.
\end{lemma}

\begin{lemma}\label{lem3}
Let $F\subset\C$ be two algebraically closed fields and $\bG,\bG'\subset\GL_{k,F}$ two connected reductive subgroups over $F$.
If $\bG\times_F\C$ and $\bG'\times_F\C$ are conjugate in $\GL_{k,\C}$, then $\bG$ and $\bG'$ are conjugate in $\GL_{k,F}$.
\end{lemma}

Let $\bT^{\ss}_\Q\subset\bT_\Q\subset\GL_{k,\Q}$ be the subtori in Theorem \ref{bicharQ}. Then up to conjugation we may assume
\begin{equation}\label{e1}
\bT^{\ss}_\Q\times_\Q\Q_\ell\subset\bT_\Q\times_\Q\Q_\ell\subset\GL_{k,\Q_\ell}
\end{equation}
is a formal bi-character of the algebraic monodromy group $\bG_\ell$ for all sufficiently large $\ell$.
Let $M\in\GL_k(\bar\Q)$ be an invertible matrix such that $\phi_M(\bT_\Q\times_\Q\bar\Q):=M(\bT_\Q\times_\Q\bar\Q )M^{-1}$ is diagonal in $\GL_{k,\bar\Q}$.
This matrix is chosen once and for all.
Then $\phi_M(\bT^{\ss}_\Q\times_\Q\bar\Q)\subset \phi_M(\bT_\Q\times_\Q\bar\Q)$ 
is defined over $\Q$ because they are split subtori of the diagonal. We obtain a chain of algebraic groups
 \begin{equation}\label{e2}
\phi_M(\bT^{\ss}_\Q):=\phi_M(\bT^{\ss}_\Q\times_\Q\bar\Q)\subset \phi_M(\bT_\Q):=\phi_M(\bT_\Q\times_\Q\bar\Q)\subset \GL_{k,\Q}.
\end{equation}

\begin{prop}\label{prop4}
There exists a connected split reductive subgroup $\bG_{\Q}^{\mathrm{sp}}$ of $\GL_{k,\Q}$ admitting (\ref{e2}) as a formal bi-character such that $\bG_{\Q}^{\mathrm{sp}}\times_\Q\bar\Q_\ell$ and $\bG_\ell\times_{\Q_\ell}\bar\Q_\ell$ are conjugate in $\GL_{k,\bar\Q_\ell}$ for all sufficiently large $\ell$.
\end{prop}

\begin{proof}
Embed $\bar\Q\subset\bar\Q_{\ell}\subset\C$ for $\ell\gg0$. 
Then $M\in\GL_k(\bar\Q)\subset\GL_k(\bar\Q_{\ell})\subset\GL_k(\C)$ and the base change of (\ref{e2}) to $\C$
\begin{equation}\label{e3}
\phi_M(\bT^{\ss}_\Q)\times_{\Q}\C\subset \phi_M(\bT_\Q)\times_{\Q}\C\subset\GL_{k,\C}
\end{equation}
is a formal bi-character of $\phi_M(\bG_{\ell}\times_{\Q_{\ell}}\C)$. Let $\bG_\Q^{\mathrm{sp}}$ be the connected split reductive group over $\Q$ such that 
$\bG_\Q^{\mathrm{sp}}\times_\Q\C$ and $\bG_{\ell}\times_{\Q_{\ell}}\C$ are isomorphic (Theorem \ref{rdisom}(i)). Then $\bG_\Q^{\mathrm{sp}}$ can be embedded into $\GL_{k,\Q}$ such that $\bG_\Q^{\mathrm{sp}}\times_\Q\C$ and $\bG_{\ell}\times_{\Q_{\ell}}\C$ are conjugate in $\GL_{k,\C}$ by Lemma \ref{lem2} and the fact that any $\bar \Q$-representation of $\bG_\Q^{\mathrm{sp}}\times_\Q\bar\Q$ can be descended to a $\Q$-representation of $\bG_\Q^{\mathrm{sp}}$ \cite[Theorem 2.5]{Ti71}. Hence, $\bG_\Q^{\mathrm{sp}}\times_\Q\C$ and $\bG_{\ell}\times_{\Q_\ell}\C$ are also conjugate in $\GL_{k,\C}$ for all sufficiently large $\ell$ and 
any embedding $\bar\Q_\ell\subset\C$ by Theorem \ref{rdr}.
This implies $\bG_\Q^{\mathrm{sp}}\times_\Q\bar\Q_\ell$ and $\bG_{\ell}\times_{\Q_\ell}\bar\Q_\ell$ are conjugate in $\GL_{k,\bar\Q_\ell}$ by $\bar\Q_\ell\subset\C$ and Lemma \ref{lem3}. 
Since $\bG_\Q^{\mathrm{sp}}$ is split and (\ref{e3}) is a formal bi-character of $\phi_M(\bG_{\ell}\times_{\Q_{\ell}}\C)$,  we may assume (\ref{e2}) is a formal bi-character of $\bG_\Q^{\mathrm{sp}}$. 
\end{proof}

\begin{definition}\label{d4.5}
For all sufficiently large $\ell$, define the following notations.
\begin{enumerate}[($i$)]
\item $\bT_\Q^{\mathrm{sp}}:=\phi_M(\bT_\Q)$
\item $\bT_\Q^{\mathrm{ssp}}:=\phi_M(\bT^{\ss}_\Q)$
\item $\bG_{\Q_\ell}^{\mathrm{sp}}:=\bG_\Q^{\mathrm{sp}}\times_\Q\Q_\ell$
\item $\bT_{\Q_\ell}^{\mathrm{sp}}:=\bT_\Q^{\mathrm{sp}}\times_\Q\Q_\ell$
\item $\bT_{\Q_\ell}^{\mathrm{ssp}}:=\bT_\Q^{\mathrm{ssp}}\times_\Q\Q_\ell$
\end{enumerate}
\end{definition}

For any non-Archimedean valuation $\bar v$ on $\bar\Q$ extending the $\ell$-adic valuation on $\Q$, there exists an
embedding $i_{\bar v}:\bar\Q\hookrightarrow \bar\Q_\ell$ such that the restriction of 
the natural non-Archimedean valuation of $\bar\Q_\ell$ to $\bar\Q$ is
 $\bar v$. Then 
we obtain a monomorphism $f_{\bar v}:\Gal_{\Q_\ell}\hookrightarrow\Gal_\Q$ such that 
the image of $f_{\bar v}$ is the decomposition subgroup of $\Gal_\Q$ at $\bar v$.

\begin{lemma}\label{prop1}
For any non-Archimedean valuation $\bar v$ on $\bar\Q$ extending the $\ell$-adic valuation on $\Q$, 
there is a natural morphism $h_{\bar v}$ of the diagram (\ref{ses3}) for $(\bG_\Q^{\mathrm{sp}},\bT_\Q^{\mathrm{sp}})$ to the diagram (\ref{ses3}) for $(\bG_{\Q_\ell}^{\mathrm{sp}},\bT_{\Q_\ell}^{\mathrm{sp}})$ satisfying the following.
\begin{enumerate}[(i)]
\item The morphism $h_{\bar v}$ is compatible with $f_{\bar v}:\Gal_{\Q_\ell}\to\Gal_\Q$ in the sense of \cite[Chapter 1 $\mathsection2.4$]{Se97}, i.e., when we view the diagram (\ref{ses3}) for $\Q$ 
as a $\Gal_{\Q_\ell}$-diagram via $f_{\bar v}$, then $h_{\bar v}$ is a $\Gal_{\Q_\ell}$-morphism of $\Gal_{\Q_\ell}$-diagrams.
\item The maps 
$h_{\bar v}:\mathrm{Out}_{\bar\Q}\bG^{\mathrm{sp}}_\Q\to\mathrm{Out}_{\bar\Q_\ell}\bG^{\mathrm{sp}}_{\Q_\ell}$ and $h_{\bar v}:\Omega_{\bar\Q}\to\Omega_{\bar\Q_\ell}$ are isomorphisms.
\end{enumerate}
\end{lemma}

\begin{proof}
The embedding $i_{\bar v}:\bar\Q\hookrightarrow \bar\Q_\ell$ identifies the following 
natural inclusions and canonical isomorphisms:
\begin{equation}
\begin{aligned}\label{comwith}
\bT^{\mathrm{sp}}_\Q(\bar\Q)&\subset\bT_{\Q_\ell}^{\mathrm{sp}}(\bar\Q_\ell);\\
\bN_\Q/\bC_\Q(\bar\Q)&\subset\bN_{\Q_\ell}/\bC_{\Q_\ell}(\bar\Q_\ell);\\
\Aut_{\bar\Q,\bT_\Q^{\mathrm{sp}}}\bG_\Q^{\mathrm{sp}}&\subset\Aut_{\bar\Q_\ell,\bT_{\Q_\ell}^{\mathrm{sp}}}\bG_{\Q_\ell}^{\mathrm{sp}};\\
\Aut_{\bar\Q}\bT_\Q^{\mathrm{sp}}&\cong\Aut_{\bar\Q_\ell}\bT_{\Q_\ell}^{\mathrm{sp}};\\
\mathrm{Weyl~group~for}~\bG_\Q^{\mathrm{sp}}\times_\Q\bar\Q&\cong \mathrm{Weyl~group~for}~\bG^{\mathrm{sp}}_{\Q_\ell}\times_{\Q_\ell}\bar\Q_\ell,
\end{aligned}
\end{equation}
which induce two inclusions:
\begin{align*}
\mathrm{Out}_{\bar\Q}\bG_\Q^{\mathrm{sp}}&\subset\mathrm{Out}_{\bar\Q_\ell}\bG_{\Q_\ell}^{\mathrm{sp}};\\
\Omega_{\bar\Q}&\subset\Omega_{\bar\Q_\ell},
\end{align*}
where the first one is an isomorphism by Theorem \ref{rdisom}(ii).
Hence, the second one is also an isomorphism by the isomorphism of the outer automorphism groups, the isomorphism of the Weyl groups, 
and  the exactness of 
the bottom row of the diagram (\ref{ses3}). These inclusions and isomorphisms comprise $h_{\bar v}$, which is
compatible with $f_{\bar v}$ because (\ref{comwith}) is compatible with $f_{\bar v}$.
\end{proof}

We have the $\bar\Q$-isomorphism
$\phi_M: \bT_\Q\times_\Q\bar\Q\to \bT_\Q^{\mathrm{sp}}\times_\Q\bar\Q$. For all sufficiently large $\ell$ and $\bar v$ as above, 
$M_{\bar v}:=i_{\bar v}(M)$ belongs to $\GL_k(\bar\Q_\ell)$ and we obtain a $\bar\Q_\ell$-isomorphism
$\phi_{M_{\bar v}}: \bT_\Q\times_\Q\bar\Q_\ell\to \bT_{\Q_\ell}^{\mathrm{sp}}\times_{\Q_\ell}\bar\Q_\ell$.
The corollary below follows directly from (\ref{action1}) and Lemma \ref{prop1}.

\begin{cor}\label{cor1}
Let 
\begin{equation}\label{e4}
\begin{aligned}
(c_\sigma)&:=(c_\sigma=\phi_M(\phi^{-1}_{\sigma M}) :\sigma\in\Gal_\Q)\in Z^1(\Q,\Aut_{\bar\Q}\bT_\Q^{\mathrm{sp}})\\
(c_{\bar v,\sigma})&:=(c_{\bar v,\sigma}=\phi_{M_{\bar v}}(\phi^{-1}_{\sigma M_{\bar v}}) :\sigma\in\Gal_{\Q_\ell})\in Z^1(\Q_\ell,\Aut_{\bar\Q_\ell}\bT_{\Q_\ell}^{\mathrm{sp}})
\end{aligned}
\end{equation}
be the cocycles whose cohomology classes  represent $\bT_\Q$ and $\bT_\Q\times_\Q\Q_\ell$ respectively.
Then $c_{\bar v,\sigma}=h_{\bar v}\circ c_\sigma \circ f_{\bar v}$ for all $\sigma\in\Gal_{\Q_\ell}$.
\end{cor}

\begin{prop}\label{prop5}
For all sufficiently large $\ell$ and $\bar v$ as above, there exists an 
 isomorphism 
\begin{equation*}
\phi_{\bar v}:(\bG_\ell\times_{\Q_\ell}\bar\Q_\ell,\bT_\Q\times_\Q\bar\Q_\ell)\to (\bG_{\Q_\ell}^{\mathrm{sp}}\times_{\Q_\ell}\bar\Q_\ell,\bT_{\Q_\ell}^{\mathrm{sp}}\times_{\Q_\ell}\bar\Q_\ell)
\end{equation*}
such that the cocycle 
$$(c_{\bar v,\sigma}'):=(c_{\bar v,\sigma}'=\phi_{\bar v}\sigma\phi^{-1}_{\bar v}\sigma^{-1}: \sigma\in\Gal_{\Q_\ell})
\in Z^1(\Q_\ell,\Aut_{\bar \Q_\ell,\bT_{\Q_\ell}^{\mathrm{sp}}}\bG_{\Q_\ell}^{\mathrm{sp}})$$
representing $(\bG_\ell,\bT_\Q\times_\Q\Q_\ell)$ (Remark \ref{GTform}) satisfies the equation in 
$Z^1(\Q_\ell,\Aut_{\bar\Q_\ell}\bT_{\Q_\ell}^{\mathrm{sp}})$:
\begin{equation*}
\mathrm{Res}(c_{\bar v,\sigma}')=(c_{\bar v,\sigma}),
\end{equation*}
where $\mathrm{Res}$ is the map in the diagram (\ref{ses3}), 
%$\mathrm{Res}(c_{\bar v,\sigma}')$ is identified as an element of $Z^1(\Q_\ell,\Aut_{\bar\Q_\ell}\bT_{\Q_\ell}^{\mathrm{sp}})$ 
$\Omega_{\bar\Q_\ell}\subset \Aut_{\bar\Q_\ell}\bT_{\Q_\ell}^{\mathrm{sp}}$ in Proposition \ref{4.3}, and $(c_{\bar v,\sigma})$ in (\ref{e4}).
\end{prop}

\begin{proof}
It suffices to find an isomorphism $\phi_{\bar v}$ such that the restriction of $\phi_{\bar v}$
to $\bT_\Q\times_\Q\bar\Q_\ell$ is $\phi_{M_{\bar v}}$.
By Proposition \ref{prop4}, there exists $P_{\bar v}\in\GL_k(\bar\Q_\ell)$ such that 
$$\phi_{P_{\bar v}}(\bG_\ell\times_{\Q_\ell}\bar\Q_\ell,\bT_\Q\times_\Q\bar\Q_\ell):=
(P_{\bar v}(\bG_\ell\times_{\Q_\ell}\bar\Q_\ell)P_{\bar v}^{-1},P_{\bar v}(\bT_\Q\times_\Q\bar\Q_\ell)P_{\bar v}^{-1})
=(\bG_{\Q_\ell}^{\mathrm{sp}}\times_{\Q_\ell}\bar\Q_\ell,\bT_{\Q_\ell}^{\mathrm{sp}}\times_{\Q_\ell}\bar\Q_\ell).
$$
Write $P_{\bar v}=N_{\bar v}M_{\bar v}$ in $\GL_k(\bar\Q_\ell)$. 
Then by Proposition \ref{prop4} again, $\phi_{M_{\bar v}}(\bG_\ell\times_{\Q_\ell}\bar\Q_\ell)$ and $\bG_{\Q_\ell}^{\mathrm{sp}}\times_{\Q_\ell}\bar\Q_\ell$ have the same formal bi-character 
$$\bT_{\Q_\ell}^{\mathrm{ssp}}\times_{\Q_\ell}\bar\Q_\ell\subset \bT_{\Q_\ell}^{\mathrm{sp}}\times_{\Q_\ell}\bar\Q_\ell\subset\GL_{k,\bar\Q_\ell}.$$
Since the algebraic monodromy groups satisfy Hypothesis \ref{A}, the root data of 
$(\phi_{M_{\bar v}}(\bG_\ell\times_{\Q_\ell}\bar\Q_\ell),\phi_{M_{\bar v}}(\bT_\Q\times_\Q\bar\Q_\ell))$ and 
$(\bG_{\Q_\ell}^{\mathrm{sp}}\times_{\Q_\ell}\bar\Q_\ell,\bT_{\Q_\ell}^{\mathrm{sp}}\times_{\Q_\ell}\bar\Q_\ell)$ are identical
by embedding $\bar\Q_\ell$ into $\C$ and applying Theorem \ref{rdr}.
Let this root datum be $\Psi_{\bar v}$. Then the isomorphism $\phi_{N_{\bar v}}$ between the two pairs
$(\phi_{M_{\bar v}}(\bG_\ell\times_{\Q_\ell}\bar\Q_\ell),\phi_{M_{\bar v}}(\bT_{\Q}\times_\Q\bar\Q_\ell))$ and 
$(\bG_{\Q_\ell}^{\mathrm{sp}}\times_{\Q_\ell}\bar\Q_\ell,\bT_{\Q_\ell}^{\mathrm{sp}}\times_{\Q_\ell}\bar\Q_\ell)$
induces an automorphism $f(\phi_{N_{\bar v}})$ of $\Psi_{\bar v}$.
By Theorem \ref{rdisom}(ii), there exists an automorphism $\Lambda_{\bar v}$ of $(\bG_{\Q_\ell}^{\mathrm{sp}}\times_{\Q_\ell}\bar\Q_\ell,\bT_{\Q_\ell}^{\mathrm{sp}}\times_{\Q_\ell}\bar\Q_\ell)$ such that the induced map $f(\Lambda_{\bar v})$ on $\Psi_{\bar v}$ is equal to
$f(\phi_{N_{\bar v}})^{-1}$. Therefore, $\phi_{\bar v}:=\Lambda_{\bar v}\circ\phi_{P_{\bar v}}$ is the desired isomorphism.
\end{proof}

\begin{customthm}{\ref{main}(ii)}
\textit{Let $\{\Phi_\ell\}_\ell$ be the system (\ref{1}) and $\bG_\ell$ the connected algebraic monodromy group 
of $\Phi_\ell^{\ss}$ for all $\ell$.
Suppose Hypothesis \ref{A} is satisfied.
Then there exists 
a connected quasi-split reductive group $\bG_\Q$ defined over $\Q$ such that for all sufficiently 
large $\ell$,}
\begin{equation*}
\bG_\ell\cong\bG_\Q\times_\Q\Q_\ell.
\end{equation*}
\end{customthm}

\begin{proof}
Let $\Omega_{\bar\Q}$ (resp. $\Omega_{\bar\Q_\ell}$) be the group defined in Proposition \ref{4.3} for $(\bG_{\Q}^{\mathrm{sp}},\bT_{\Q}^{\mathrm{sp}})$ (resp. $(\bG_{\Q_\ell}^{\mathrm{sp}},\bT_{\Q_\ell}^{\mathrm{sp}})$).
From now on we assume $\ell$ is sufficiently large and $\bar v$ is a valuation of $\bar\Q$ extending the $\ell$-adic valuation of $\Q$. Then the cocycle $(c_{\bar v,\sigma})$ in (\ref{e4}) belongs to $Z^1(\Q_\ell,\Omega_{\bar\Q_\ell})$
by Proposition \ref{prop5}. 
We view $(c_\sigma)$ (resp. $(c_{\bar v,\sigma})$) as a homomorphism from $\Gal_\Q$ to $\Aut_{\bar\Q_\ell}\bT^{\mathrm{sp}}_{\Q_\ell}$ 
(resp. $\Gal_{\Q_\ell}$ to $\Omega_{\bar\Q_\ell}$) since the Galois action on the target group is trivial.
Since the image of $(c_\sigma)$ is finite, it is unramified except at finitely many primes. 
Hence, its image is determined by the image of the Frobenius elements (i.e., the image of $c_\sigma\circ f_{\bar v}$) by the Chebotarev density theorem.
Since $c_{\bar v,\sigma}=h_{\bar v}\circ c_\sigma \circ f_{\bar v}$ (Corollary \ref{cor1}), $\mathrm{Im}(c_{\bar v,\sigma})\subset\Omega_{\bar\Q_\ell}$, and 
$h_{\bar v}:\Omega_{\bar\Q}\to\Omega_{\bar\Q_\ell}$ is an isomorphism (Lemma \ref{prop1}) for all $\bar v|\ell$ and sufficiently large $\ell$,
the image of cocycle $(c_\sigma)$ is contained in 
 $\Omega_{\bar\Q}$, i.e.,  $(c_\sigma)\in Z^1(\Q,\Omega_{\bar\Q})$. 
Hence by the diagram (\ref{ses3}), the cocycle $(c_\sigma)$ maps to the cohomology class 
$[\bar c_\sigma]\in H^1(\Q,\mathrm{Out}_{\bar\Q}\bG_\Q^{\mathrm{sp}})$ 
and corresponds to a unique connected reductive quasi-split group $\bG_\Q$ over $\Q$ by Proposition \ref{4.1} and Theorem \ref{4.2}.
Let $[\bar c_{\bar v,\sigma}]$\footnote{For $\ell\gg0$, this class is just the image 
of the class $[(\bG_\ell,\bT_\Q\times_\Q\Q_\ell)]\in H^1(\Q_\ell,\Omega_{\bar\Q_\ell})$ which is independent of $\bar v$.} 
be the cohomology class of the cocycle $(\bar c_{\bar v,\sigma})\in Z^1(\Q_\ell,\mathrm{Out}_{\bar\Q_\ell}\bG_{\Q_\ell}^{\mathrm{sp}})$. 
Since $\bG_\ell$ is quasi-split (Corollary \ref{maxcor}),
 $[\bar c_{\bar v,\sigma}]$
corresponds to $\bG_\ell$ by construction (Proposition \ref{prop5}), Proposition \ref{4.1}, and Theorem \ref{4.2}.
Since $h_{\bar v}\circ\bar c_\sigma\circ f_{\bar v} = \bar c_{\bar v,\sigma}$ (in $\mathrm{Out}_{\bar\Q_\ell}\bG_{\Q_\ell}^{\mathrm{sp}}$) by Corollary \ref{cor1} and both $\bG_\Q$ and $\bG_\ell$ are quasi-split, we obtain
$\bG_\ell\cong\bG_\Q\times_\Q\Q_\ell$ by Theorem \ref{4.2}.
\end{proof}

\begin{remark}
\begin{enumerate}[($i$)]
\item Suppose $\bT_\Q$ is the projection of the Frobenius torus $\bT'_\Q=\bT'_{\bar v,\ell'}$ for some prime $\ell'$ and some $\bar v\in\Sigma_{\bar K}$ (see the proofs of
Corollary \ref{FMTC3} and Theorem \ref{bicharQ}). Then $\bG_\ell$ contains a conjugation of $\bT_\Q\times_\Q\Q_\ell$ as a maximal torus and is an inner twist of 
the quasi-split group $\bG_\Q\times_\Q\Q_\ell$ for all prime $\ell$ not divisible by $\bar v$.
\item Assume for simplicity that $\bG_\ell$ is of type $A_n$ and is an inner twist of $\bG_\Q\times_\Q\Q_\ell$ for all $\ell$. 
Constructing a common $\Q$-form of $\bG_\ell$ for all $\ell$ amounts to solve for a $\Q$-central simple algebra $A$ 
with prescribed local invariants $\tau_\ell\in\Q/\Z$ corresponding to the inner twists
for all primes $\ell$. By the fundamental exact sequence of Brauer groups for $\Q$
\begin{equation*}
1\to \mathrm{Br}(\Q)\to \bigoplus_v\mathrm{Br}(\Q_v)\to \Q/\Z\to 1,
\end{equation*}
finding such an algebra $A$ is equivalent to showing that the sum of these local invariants at the finite places 
belongs to $\Z/2\Z$. Since the only thing we know is $\tau_\ell=0$ 
for all sufficiently large $\ell$ (Theorem \ref{maxcor}), finding a $\Q$-form for all $\ell$ needs extra information. 
\item It is reasonable to ask if the data $\bT_\Q\subset\GL_{k,\Q}$ (a $\Q$-form of formal character of $\bG_\ell\subset\GL_{k,\Q_\ell}$), 
the $\ell$-independence of absolute root datum (Theorem \ref{rdr}), and Tits's theory of descending representations \cite{Ti71} are enough to construct for all sufficiently large $\ell$ a common $\Q$-form of the embeddings $\bG_\ell\subset\GL_{k,\Q_\ell}$. We tried but did not succeed. 
\end{enumerate}
\end{remark}

\subsection{Proofs of Corollary \ref{newcor} and \ref{cor}}
Apply the constructions in Definition \ref{newgp} to $\Gamma_\ell\subset\bG_\ell(\Q_\ell)$,
 we obtain the morphisms $\pi_\ell^{\sc}:\bG_\ell^{\sc}(\Q_\ell)\to\bG_\ell^{\der}(\Q_\ell)$ 
and $\pi_\ell^{\ss}:\bG_\ell(\Q_\ell)\to\bG_\ell^{\ss}(\Q_\ell)$
and the groups $\Gamma_\ell^{\sc}$ and $\Gamma_\ell^{\ss}$ for all $\ell$.

\begin{customcor}{\ref{newcor}}
\textit{Let $\cG^{\sc}$ be
a semisimple group scheme  over $\Z[\frac{1}{N}]$ for some $N$ whose generic fiber is $\mathbf{G}_\Q^{\sc}$, where 
$\bG_\Q$ us in Theorem \ref{main}.
For all sufficiently large $\ell$, 
we have }
$$\Gamma_\ell^{\sc}\cong\cG^{\sc}(\Z_\ell).$$
\end{customcor}

\begin{proof}
Let $\bG_\Q^{\sc}$ be the universal covering group of $\bG_\Q^{\der}$. By Theorem \ref{main}(ii), we have
$\bG_\ell^{\sc}\cong\bG_\Q^{\sc}\times_\Q\Q_\ell$ for $\ell\gg0$. 
Let $\cG^{\sc}$ and $\cG^{\der}$ be 
 semisimple group schemes  over $\Z[\frac{1}{N}]$ for some $N$ whose generic fibers are
 $\mathbf{G}_\Q^{\sc}$ and $\bG_\Q^{\der}$ respectively.
The central isogeny
$$\pi_\Q^{\sc}:\bG_\Q^{\sc}\to \bG_\Q^{\der}$$ 
can be extended to a morphism of smooth affine group schemes over $\Z[\frac{1}{N'}]$ ($N'$ is some multiple of $N$)
\begin{equation}\label{Zmap}
\pi_{\Z[\frac{1}{N'}]}^{\sc}:\cG^{\sc}\times_{\Z[\frac{1}{N}]}\Z[\frac{1}{N'}]\to \cG^{\der}\times_{\Z[\frac{1}{N}]}\Z[\frac{1}{N'}].
\end{equation}
Since all hyperspecial subgroups of $\bG_\Q^{\sc}(\Q_\ell)\cong\cG^{\sc}(\Q_\ell)$ are isomorphic \cite[$\mathsection2.5$]{Ti79} and 
$$\Gamma_\ell^{\sc}\subset\bG_\Q^{\sc}(\Q_\ell)\cong\cG^{\sc}(\Q_\ell)$$ for sufficiently large $\ell$ is hyperspecial by Theorem \ref{max}, we obtain
$\Gamma_\ell^{\sc}\cong\cG^{\sc}(\Z_\ell)$ for $\ell\gg0$ by \cite[$\mathsection3.9.1$]{Ti79}.
\end{proof}

Let $\ell\geq 5$ be prime, $\bH_\ell$ a connected algebraic group defined over $\Q_\ell$, and $\Delta_\ell$ a compact subgroup of $\bH_\ell(\Q_\ell)$. Then by embedding $\bH_\ell$ into some $\GL_{n,\Q_\ell}$ and finding some $\Z_\ell$-lattice of $\Q_\ell^n$ invariant under $\Delta_\ell$, one obtains
a finite subgroup $\bar\Delta_\ell$ of $\GL_n(\F_\ell)$ by taking mod $\ell$ reduction. Then $\mathrm{Lie}_\ell\bar\Delta_\ell$ (Definition \ref{Lie})
is independent of the embedding $\bH_\ell\subset\GL_{n,\Q_\ell}$ and the mod $\ell$ reduction because the kernel of $\Delta_\ell\twoheadrightarrow\bar\Delta_\ell$ is pro-solvable. This allows us to make the following definition.

\begin{definition}\label{cf}
For any prime $\ell\geq 5$ and compact subgroup $\Delta_\ell$ of $\bH_\ell(\Q_\ell)$,
the composition factors of Lie type in characteristic
$\ell$ of $\Delta_\ell$, denoted by $\mathrm{Lie}_\ell\Delta_\ell$, is defined to be the multiset
$\mathrm{Lie}_\ell\bar\Delta_\ell$ (Definition \ref{Lie}), where the finite group $\bar\Delta_\ell$ is constructed above.
\end{definition}

\begin{lemma}\label{lem9}
Suppose $\ell\geq 5$. Then $\mathrm{Lie}_\ell\Gamma_\ell=\mathrm{Lie}_\ell\Gamma_\ell^{\sc}$.
\end{lemma}

\begin{proof}
Since the kernel of  
$$\pi_\ell^{\ss}:\Gamma_\ell\twoheadrightarrow\Gamma_\ell^{\ss}$$
 is pro-solvable, we obtain
$\mathrm{Lie}_\ell\Gamma_\ell=\mathrm{Lie}_\ell\Gamma_\ell^{\ss}$.
Since the kernel and cokernel of 
$$\pi_\ell^{\ss}\circ\pi_\ell^{\sc}:\Gamma_\ell^{\sc}\to\Gamma_\ell^{\ss}$$
 are abelian, we obtain
$\mathrm{Lie}_\ell\Gamma_\ell^{\ss}=\mathrm{Lie}_\ell\Gamma_\ell^{\sc}$. We are done.
\end{proof}

\begin{customcor}{\ref{cor}}
\textit{Let $\cG^{\der}$ be
a semisimple group scheme  over $\Z[\frac{1}{N}]$ for some $N$ whose generic fiber is $\mathbf{G}_\Q^{\der}$, where 
$\bG_\Q$ is in Theorem \ref{main}.
For all sufficiently large $\ell$, 
we have }
$$\mathrm{Lie}_\ell\bar\Gamma_\ell=\mathrm{Lie}_\ell\cG^{\der}(\F_\ell).$$
\end{customcor}

\begin{proof}
Since the mod $\ell$ representation $\phi_\ell$ ($\mathsection\ref{mod l}$) is the semisimplification 
of a mod $\ell$ reduction of $\Phi_\ell$ and $\mathrm{Lie}_\ell\bar\Gamma=\emptyset$ for any finite solvable group $\bar\Gamma$,  we obtain 
\begin{equation}\label{Lie1}
\mathrm{Lie}_\ell\bar\Gamma_\ell=\mathrm{Lie}_\ell\Gamma_\ell=\mathrm{Lie}_\ell\Gamma_\ell^{\sc}
\end{equation}
for all $\ell$ by Definition \ref{cf} and Lemma \ref{lem9}.
Since $\Gamma_\ell^{\sc}\cong\cG^{\sc}(\Z_\ell)$ for $\ell\gg0$ by Corollary \ref{newcor}, the kernel of reduction map $\cG^{\sc}(\Z_\ell)\twoheadrightarrow\cG^{\sc}(\F_\ell)$ is pro-solvable, and the kernel and cokernel of
$\pi^{\sc}_{\Z[\frac{1}{N'}]}:\cG^{\sc}(\F_\ell)\to \cG^{\der}(\F_\ell)$ (\ref{Zmap}) are abelian for $\ell\gg0$, we obtain 
\begin{equation}\label{Lie2}
\mathrm{Lie}_\ell\Gamma_\ell^{\sc}=\mathrm{Lie}_\ell\cG^{\sc}(\Z_\ell)=\mathrm{Lie}_\ell\cG^{\sc}(\F_\ell)=\mathrm{Lie}_\ell\cG^{\der}(\F_\ell).
\end{equation}
We are done by (\ref{Lie1}) and (\ref{Lie2}).
\end{proof}

\section*{Acknowledgments} 
The idea for constructing the $\Q$-form $\bG_\Q$ is inspired by Larsen-Pink \cite[$\mathsection5$]{LP95} and Pink \cite[$\mathsection6$]{Pi98}.
I would like to thank Michael Larsen and Gabor Wiese for their comments on an earlier preprint. 
I am grateful to the anonymous referees for pointing out \cite{DG62},\cite{Ma09},\cite{Vi96} to simplify $\mathsection\ref{s5}$
and many helpful comments and suggestions which have greatly improved the readability of the paper.

\bigskip

\noindent Department of Mathematics \\
VU University - Faculty of Sciences \\
De Boelelaan 1081a\\
1081 HV Amsterdam\\
The Netherlands\\

\noindent Email address: \texttt{pslnfq@gmail.com}
\end{document}